\documentclass[12pt]{article}
   \usepackage{latexsym}
   \usepackage[reqno, namelimits, sumlimits]{amsmath} 
   \usepackage{amssymb, amsthm, amsfonts} 
\newtheorem{Theorem}{Theorem}[section]

\newtheorem{theorem}[Theorem]{Theorem}
\newtheorem{lemma}[Theorem]{Lemma}
\newtheorem{proposition}[Theorem]{Proposition}
\newtheorem{corollary}[Theorem]{Corollary}
\newtheorem{remark}[Theorem]{Remark}

\newcommand{\R}{{\mathbb R}}

\newcommand{\eps}{\varepsilon}
\newcommand{\ep}{\varepsilon}
\newcommand{\al}{\alpha}
\newcommand{\de}{\delta}

\numberwithin{equation}{section}

\title{Entire and ancient solutions 
of a supercritical semilinear heat equation}
\author{%
 Peter Pol\'a\v cik\footnote{Supported in part by NSF Grant
   DMS-1856491
}
\\
\small School of Mathematics, University of Minnesota\\
\small Minneapolis, MN 55455
\\~\\
 Pavol Quittner\footnote{Supported in part by VEGA Grant 1/0347/18
 and by the Slovak Research and Development Agency 
 under the contracts No. APVV-14-0378 and APVV-18-0308} 
\\
\small Department of Applied Mathematics and Statistics, 
       Comenius University,\\
\small Mlynsk\'a dolina, 84248 Bratislava, Slovakia} 

\date{}  

\begin{document}
 
\maketitle

\begin{quote}{\small
    {\bfseries Abstract.}
    We consider the semilinear heat equation
    $ u_t=\Delta u+u^p$ on $\R^N$. Assuming that $N\ge 3$ and
    $p$ is greater than the
    Sobolev critical exponent $(N+2)/(N-2)$, we examine 
    entire solutions (classical solutions defined for all $t\in \R$)
    and ancient solutions (classical solutions defined on
    $(-\infty,T)$ for some $T<\infty$).  We prove a new Liouville-type
    theorem saying that if $p$ is greater than the Lepin exponent
    $p_L:=1+6/(N-10)$ ($p_L=\infty$ if $N\le 10$), then all positive
    bounded radial entire solutions are steady states. The theorem is not
    valid without the assumption of radial symmetry; in other ranges
    of supercritical $p$ it is known not to be valid even in the class of
    radial solutions. Our other results 
    include classification theorems for nonstationary entire
    solutions (when they exist) and ancient solutions, as well as some
    applications in the theory of blowup of solutions. 
} 
 \end{quote}



\tableofcontents 
%-----------------------------------------------------------------
\section{Introduction}
\label{intro}

Entire and ancient solutions play an important role 
in studies of singularities and long-time behavior 
of solutions of many evolution problems. In that vein,
of prominent importance are entire and ancient solutions of
some specific equations which serve as 
scaling limits of many other equations with
a given structure.

In this paper, we consider the semilinear heat equation 
\begin{equation} \label{eq-Fuj}
u_t=\Delta u+u^p,
\end{equation}
where $u=u(x,t)>0$,  $x\in\R^N$, and $p>1$. 
We investigate
positive classical solutions of the problems
\begin{equation} \label{eq-entire}
u_t=\Delta u+u^p, \qquad x\in\R^N,\ t\in(-\infty,\infty),
\end{equation}
(entire solutions of \eqref{eq-Fuj}), and
\begin{equation} \label{eq-ancient}
u_t=\Delta u+u^p, \qquad x\in\R^N,\ t\in(-\infty,T),
\end{equation}
where $T<\infty$ (ancient solutions of \eqref{eq-Fuj}).

 Note that equation \eqref{eq-Fuj} is invariant under
 the scaling
 $$u(x,t)\mapsto\lambda^{2/(p-1)}u(\lambda x,\lambda^2t).$$
With respect to the same scaling, \eqref{eq-Fuj}  can be considered as 
the scaling limit of a large class of equation whose nonlinearities
have polynomial growth, such as equations of the form
\begin{equation} \label{eq-g}
u_t=\Delta u+u^p+g(u),\qquad x\in\R^N,
\end{equation}
where $g$ is a continuous function with $\lim_{u\to\infty}u^{-p}g(u)=0$. 
More specifically, applying the above scaling to equation \eqref{eq-g}
and taking formally $\lambda\to \infty$, one obtains equation \eqref{eq-Fuj}. Of
course, the connection between \eqref{eq-g} and \eqref{eq-Fuj} is not
just formal; it is well known that with  good understanding of
\eqref{eq-Fuj},  in particular of its entire and ancient solutions,
one can draw interesting conclusions about solutions of the Cauchy
problem for \eqref{eq-g} (Corollary \ref{cor-subcrit} below is an
example of this).   

We are mainly interested in
radially symmetric solutions
of \eqref{eq-entire} and \eqref{eq-ancient}. 
If no confusion seems likely,
we will often consider a radial solution $u$ as a function
of $r:=|x|$ and $t$, i.e.\ $u=u(r,t)$.

The simplest entire solutions are steady states.
Positive steady states of \eqref{eq-entire}
exist if and only if $p\geq p_S$, where 
$$p_S:=\begin{cases}
     \frac{N+2}{N-2} &\hbox{ if }N>2, \\
     +\infty &\hbox{ if }N\leq2
   \end{cases} $$
is the Sobolev exponent (see \cite{GS81}, \cite{CL91} or \cite{QS07}).
If $p\geq p_S$, then radial positive steady states
form a one-parameter family $\{\phi_{\alpha}\}_{\alpha>0}$,
where $\phi_\alpha(0)=\alpha$. 
These solutions are ordered---that is,
$\phi_\alpha<\phi_\beta$ for $\alpha<\beta$---if and only if $p\geq
p_{JL}$, where
$$p_{JL}:=\begin{cases}
     1+4\frac{N-4+2\sqrt{N-1}}{(N-2)(N-10)} 
        &\hbox{ if }N>10, \\
     +\infty &\hbox{ if }N\leq10,
   \end{cases} $$
see \cite{W93} or \cite{QS07}.
   Ordered or not, the family $\{\phi_{\alpha}\}_{\alpha>0}$
   approaches as $\alpha \to \infty$ the singular steady state
$$ \phi_\infty(x):=L|x|^{-2/(p-1)}, \qquad
L:=\left(\frac2{(p-1)^2}((N-2)p-N)\right)^{\frac1{p-1}},$$
which has a special role in this paper. It
exists whenever $p(N-2)>N$.

In regard to time-dependent entire solutions, denoting
$$p^*:=\begin{cases}
     \frac{N(N+2)}{(N-1)^2} &\hbox{ if }N>2, \\
     +\infty &\hbox{ if }N\leq2,
   \end{cases} $$
the following Liouville-type theorem is known (see \cite{PQS07,BV98,Q16}):

%--------------------------------------------------
\begin{theorem} \label{thm-subcrit}
If $p<p_S$, then  \eqref{eq-entire} does not possess positive radial solutions.
If $p<p^*$, then  \eqref{eq-entire} does not possess any positive solutions.
\end{theorem}

Nonexistence of positive (non-radial) solutions
of \eqref{eq-entire} for $p\in[p^*,p_S)$ is still an open problem.
On the other hand, a nonexistence result for sign-changing radial solutions
has been obtained in \cite{BPQ11}.

Theorem~\ref{thm-subcrit} 
has a number of interesting applications
in equations 
\eqref{eq-Fuj}, \eqref{eq-g}, and 
even more general problems \cite{PQS07}.
As an illustration, we just state the following
optimal universal estimate for positive solutions of
\eqref{eq-g} on any time interval $(\tau,T)$ (see \cite[Theorem~3.1]{PQS07}).

\begin{corollary} \label{cor-subcrit}
Assume $g$ is a continuous function such that $u^{-p}g(u)\to0$ as $u\to\infty$
and  let $u$ be a positive solution of \eqref{eq-g} on an interval
$(\tau,T)$. Assume that either $u$ is radial and $p<p_S$,
or $p<p^*$. Then  
\begin{equation} \label{u-decay-subcrita}
\|u(\cdot,t)\|_\infty\leq C(1+(t-\tau)^{-1/(p-1)}+(T-t)^{-1/(p-1)})
\quad\hbox{for all }\quad 
t\in(\tau,T),
\end{equation}
where $C=C(g,n)$ is a constant independent of $u$, $\tau$, and $T$. If
$g\equiv 0$, then the following stronger version of \eqref{u-decay-subcrita} 
holds:
\begin{equation} \label{u-decay-subcritb}
\|u(\cdot,t)\|_\infty\leq C((t-\tau)^{-1/(p-1)}+(T-t)^{-1/(p-1)})
\quad\hbox{for all }\quad
t\in(\tau,T).
\end{equation}
\end{corollary}
Since $C$ is independent of $\tau$, taking $\tau\to-\infty$,
we obtain from \eqref{u-decay-subcritb}
the following estimates for ancient solutions of \eqref{eq-Fuj}:
\begin{equation} \label{u-decay-subcrit}
\|u(\cdot,t)\|_\infty\leq C(T-t)^{-1/(p-1)}
\quad\hbox{for all }\quad  
t\in(-\infty,T).
\end{equation}
For ancient solutions satisfying \eqref{u-decay-subcrit}
the following classification theorem
has been proved in  \cite{MZ98}:  

\begin{theorem} \label{thm-MZ}
Let $p<p_S$ and $u$ be a positive solution of \eqref{eq-ancient}
satisfying
\begin{equation} \label{u-Tdecay}
\|u(\cdot,t)\|_\infty\leq C(T-t)^{-1/(p-1)} \quad\hbox{as}\quad t\to-\infty.
\end{equation}
Then there exists $T^*\geq T$ such that 
$u(x,t)=\kappa(T^*-t)^{-1/(p-1)}$, where 
$$\kappa:=(p-1)^{-1/(p-1)}.$$
\end{theorem}
(In this theorem and below, we use 
$C$, $C_1$ etc., to denote constants independent of the solution in
question.) 
Thus,  Corollary~\ref{cor-subcrit} in conjunction with
Theorem~\ref{thm-MZ} shows that the only positive radial ancient solutions
are the (spatially constant) ancient solutions
of the equation $\dot\xi=\xi^p$ (if $p<p^*$, the word ``radial'' can be
omitted in  this statement). Theorem~\ref{thm-MZ} has other
interesting and important consequences 
in the study of the blowup behavior of solutions
of \eqref{eq-Fuj}, which can be found in \cite{MZ98}.

The above results are all concerned with the subcritical case
$p<p_S$. Of course, in the critical or supercritical cases,
the existence of positive radial steady states has 
to be taken into account in the formulation of any
Liouville-type theorems or
problems. A first natural question is whether there are
any positive entire solutions  other than the steady states.   
In some cases, this question has been answered in the negative, but
only when rather severe extra bounds on the solutions are imposed.  
Namely, the following Liouville-type results 
are known (see \cite[Theorem 2.4]{FY11}
and \cite[Theorem 1.2]{PY05}).

\begin{theorem} \label{thmLFPY}
Let  $u$ be a nonnegative solution of \eqref{eq-ancient}.
\begin{itemize}
\item[(i)] Assume $p_S\leq p<p_{JL}$ and $u(\cdot,t)\leq \phi_\infty$
for all $t\leq T$. Then $u\equiv0$.
\item[(ii)] Assume $p>p_{JL}$ and 
$\phi_\alpha\leq u(\cdot,t)\leq \phi_\infty$
for some $\alpha>0$ and all $t\leq T$. 
Then $u(\cdot,t)\equiv\phi_\gamma$ for some $\gamma\geq\alpha$.
\end{itemize}
\end{theorem}

Without the extra bounds, these results are 
not valid, at least in the range  $p_S\leq p<p_L$, where
$$p_L:=\begin{cases}
     1+\frac6{N-10}  &\hbox{ if }N>10, \\
     +\infty &\hbox{ if }N\leq10
   \end{cases} $$
is the critical exponent for the existence of
positive bounded non-constant radial steady states
of a rescaled equation (see \eqref{eq-v} below).
Notice that $p_L>p_{JL}$ if $N>10$.
Positive radial bounded solutions
of \eqref{eq-entire} which do depend on time
are provided by the following
results of \cite{FY11}.

\begin{theorem} \label{thm-FY}
\strut\hbox{\rm(i)}
If $p_S<p<p_L$, then there exists a positive radial bounded solution $u$ of
\eqref{eq-entire} satisfying $\lim_{|t|\to\infty}\|u(\cdot,t)\|_\infty=0$
(i.e.~$u$ is a homoclinic solution to the trivial steady state).
In addition, given $T\in\R$, $u$ also satisfies \eqref{u-Tdecay}.  

\strut\hbox{\rm(ii)}
If $p_S\leq p<p_{JL}$ and $\phi$ is a positive radial steady state of \eqref{eq-entire},
then there exists a positive radial bounded solution of \eqref{eq-entire}
satisfying 
$$\lim_{t\to-\infty}\|u(\cdot,t)-\phi\|_\infty=\lim_{t\to\infty}\|u(\cdot,t)\|_\infty=0$$
(i.e.~$u$ connects $\phi$ to zero).
\end{theorem}

With the above results, the problem of the existence
of positive radial entire (nonstationary) solutions
is settled for all $p<p_L$.  One of  the primary objectives
of our present study is to address the problem in the range $p>p_L$. 
We have the following result, the main Liouville-type theorem
of this paper.  

\begin{theorem} \label{thmL1}
Assume $p>p_L$. Then any positive radial bounded solution of
\eqref{eq-entire} is a steady state.
\end{theorem}

The  proof of this theorem is given in Section \ref{sec-proofs};
as it is rather involved, we precede it by an informal outline. 

Theorem \ref{thmL1} is not valid without the assumption of radial
symmetry. Indeed,  as indicated in a remark following Theorem~2.1 in
\cite{FY11}, one can find nontrivial entire solutions by extensions
of solutions in lower dimensions. To make this remark more precise,
fix any $p>p_L$. Then 
one can always find an integer $j\in \{3,\dots,N-1\}$ such that
$p$ is between $p_S(j)$ and $p_L(j)$, the Sobolev and Lepin exponents
in dimension $j$. Take now an entire solution $u(\tilde x,t)$, $\tilde
x\in\R^j$, as provided by  
Theorem~\ref{thm-FY}(i). 
Viewing $u$ as a function of $t$ and $x$,
constant in the last $N-j$ variables, we obtain a positive bounded
nonstationary entire solution of \eqref{eq-Fuj}.

Similarly as in the subcritical case,
the Liouville theorem for $p>p_L$ has important applications. For example, 
we will show in Section~\ref{sec-appl} that Theorem~\ref{thmL1}
can be used to prove the convergence of profiles
of both global and blowing-up solutions.

When nonstationary 
entire solutions  do exist, it is still an interesting
question if they can be classified in some way. 
Our next theorem gives a classification of entire solutions satisfying
\eqref{u-Tdecay}. Its conclusion is, in a sense,
complementary to Theorem~\ref{thm-FY}(i)
in the case $p_S<p<p_{JL}$.

\begin{theorem} \label{thm1}
If $p_S<p<p_{JL}$ and $u$ is a positive radial bounded solution of
\eqref{eq-entire} satisfying \eqref{u-Tdecay},
then \ $\lim\limits_{t\to\infty}\|u(\cdot,t)\|_\infty=0$
(hence, $u$ is a homoclinic solution to the trivial steady state).
\end{theorem}

We believe that the same statement is valid if
$p_{JL}\le p<p_{L}$, but presently we can only prove this
under an additional condition (see Remark \ref{rm:en} below).  

We now consider ancient solutions.  In order to describe our results,
we introduce the backward similarity variables
$$ y:=\frac x{\sqrt{T-t}},\ s:=-\log(T-t), $$
and the rescaled function
\begin{equation} \label{vu}
\begin{aligned}
    v(y,s) &:= (T-t)^{1/(p-1)}u(x,t) \\
           &= e^{-s/(p-1)}u(e^{-s/2}y,T-e^{-s}).
 \end{aligned}
\end{equation}
Notice that if $u$ solves \eqref{eq-ancient}, then
$v$ is an entire solution of the equation
\begin{equation} \label{eq-v}
v_s=\Delta v-\frac y2\cdot\nabla v-\frac v{p-1}+v^p, 
\qquad y\in\R^N,\ s\in(-\infty,\infty).
\end{equation}

Problem \eqref{eq-v} has a positive constant steady state $v\equiv\kappa$
for all $p>1$ and the singular steady state $\phi_\infty$ whenever
$p(N-2)>N$. Positive bounded non-constant radial steady states of \eqref{eq-v}
exist if $p\in(p_S,p_L)$, 
while such solutions do not exist if $p>p_L$,
see \cite{L90,M09} and references therein.
In the case $p=p_L$, the nonexistence is stated in the main result of
\cite{M10}, however the proof given there
contains a gap, which does not seem to have been fixed yet.

We have the following result concerning ancient solutions.

\begin{theorem} \label{thm2}
Let either $p_S<p<p_{JL}$ or $p>p_L$. Let $u$ be a positive radial solution of
\eqref{eq-ancient}, and let $v$ denote the corresponding rescaled function.

If $u$ satisfies \eqref{u-Tdecay},
then $v$ is either a positive bounded radial steady state of \eqref{eq-v} 
or connects a positive bounded radial steady state $w$ of \eqref{eq-v}
to a nonnegative bounded radial steady state $\tilde w\ne w$ of
\eqref{eq-v}:
\begin{equation}\label{connection-w}
  \lim_{s\to-\infty}v(\cdot,s)=w,\quad
  \lim_{s\to\infty}v(\cdot,s)=\tilde w,
\end{equation}
with the convergence in  $C^1_{loc}(\R^N)$.

If \eqref{u-Tdecay} fails, then $v$ connects the singular steady state
$\phi_\infty$ to a nonnegative bounded radial steady state
$\tilde w$ of \eqref{eq-v}, that is, \eqref{connection-w} holds 
with $w=\phi_\infty$, where the convergence is 
in  $C^1_{loc}(\R^N\setminus\{0\})$ in the case of $w$ and in
$C^1_{loc}(\R^N)$ in the case of $\tilde w$.
\end{theorem}

Thus, if $p_S<p<p_{JL}$ or $p>p_L$, the positive radial
ancient solutions  can be classified as heteroclinic connections
in self-similar variables, possibly with the singular backward limit.
This statement in the regular backward limit case (the first part of Theorem~\ref{thm2}) can be viewed as a (radial)
analogue of Theorem~\ref{thm-MZ} in the given supercritical ranges of
$p$. Indeed, using the rescaled function $v$, 
Theorem~\ref{thm-MZ} can be formulated as follows
(see \cite[Corollary~1.5]{MZ98}):

\begin{remark} \label{rem-MZv}{\rm
Let $p<p_S$ and $u$ be a positive solution of \eqref{eq-ancient}
satisfying \eqref{u-Tdecay}.
Then the rescaled function $v$ is either equal to the constant $\kappa$
or there exists $s_0\in\R$ such that $v(y,s)=\varphi(s-s_0)$,
where $\varphi(s):=\kappa(1+e^s)^{-1/(p-1)}$
(hence $v$ connects $\kappa$ to zero).}
\end{remark}

As an application of Theorem~\ref{thm2}, we now examine the  character
of blowup of ancient solutions. First we recall some terminology.
Let  $u$ be a positive radial solution of \eqref{eq-Fuj}
defined on a time interval $(0,T)$. This solution is
said to blow up at $t=T$ if $\|u(\cdot,t)\|_\infty\to\infty$ as $t\to T$. The
blowup is of type~I if the function
$(T-t)^{1/(p-1)}\|u(\cdot,t)\|_\infty$ stays
bounded as $t\to T$, otherwise it is of type~II.
As proved in \cite{GK87} (see also Corollary \ref{cor-subcrit} above),
type~II blowup never occurs if $p<p_S$ (this is also true
with  the assumption of radial symmetry dropped). The absence of
type~II blowup is also known for 
some classes of radial solutions (for example, radially nonincreasing solutions)
if $p_S\le p<p_{JL}$ \cite{MM04, MM09, M11a}. 
On the other hand, type~II blowup is known to occur for
some positive radial solutions 
if $p\geq p_{JL}$ (see \cite{HV94,M11,MM11,S18}). 
Let us now add the assumption
that $u$ is an ancient solution.
Our question is whether from the fact that $u$ has some ``past''
one can draw a definite conclusion about the type of
its blowup. If  $p_S<p<p_{JL}$ or $p>p_L$,
we can give a positive answer:
\begin{corollary}\label{co1}
Let either $p_S<p<p_{JL}$ or $p>p_L$.
Let $u$ be a positive radial 
solution of \eqref{eq-ancient}.
If $u$ blows up at $t=T$, then the blowup is of type~I.
\end{corollary}
This result follows directly from
Theorem~\ref{thm2}, which gives a bound on
$(T-t)^{1/(p-1)}u(\cdot,t)$ in any compact set,
and the universal estimate  \eqref{est-u}
proved in Proposition \ref{prop-ub} below,
which yields a bound on this function away
from the origin in $\R^N$.

\begin{remark}
  \label{rm:en}
  {\rm
    We conclude the introduction with a few remarks concerning
    exponents $p$ not covered by
  the above results. As previously mentioned, we expect  
  Theorem \ref{thm1} to hold in the range  $p_{JL}\le p<p_{L}$
  and can actually prove this (see Section \ref{sec-appl})
  under an additional condition. Specifically, the condition
  requires that each  classical positive radial steady
  state $w$  of \eqref{eq-v} satisfy the relation
  $E(w)<E(\phi_\infty)$, where
  $E$ is the standard energy functional for equation \eqref{eq-v} (see
  Subsection \ref{subs:energy}). 
  In Section \ref{sec-appl} we also  give some heuristics as to why the energy
  condition is plausible, but it is not clear to us if it can
  be proved by any  readily available tools. In the  borderline
  case $p=p_L$, the statement of Theorem \ref{thm1} is most likely
  void, for we do not expect any positive radial bounded solution
of (1.2) to exist---$p=p_L$ is not included in Theorem \ref{thmL1} for
several technical reasons.  In Theorem
\ref{thm2} (and Corollary~\ref{co1}), we left out the range
$p_{JL}\le p\le p_{L}$. Again,  we
believe that both statements of Theorem \ref{thm2}
are valid in this range as well, but can only prove it under
the above energy condition (see Remark \ref{rm:added}).
 }
\end{remark}

The rest of the paper is organized as follows. The
next section contains several preliminary results concerning the
energy functional for \eqref{eq-v}, zero number for differences of
solutions of equations \eqref{eq-entire}, \eqref{eq-ancient} and their
rescaled versions, and the $\alpha$- and $\omega$-limit sets of
solutions of \eqref{eq-v}. In the same preliminary section,
we also give universal a priori estimates on
radial entire and ancient solutions, and examine the relation of two radial
solutions of \eqref{eq-v} for large values of $\rho=|y|$.
The proof of Theorem \ref{thmL1} and its informal outline are
given in Section \ref{sec-proofs}. Section \ref{sec-proofs2} is
devoted to the proofs of Theorems \ref{thm1}, \ref{thm2}.
In Section \ref{sec-appl},
we discuss some applications of our results. In particular, 
we state and prove there a theorem on the convergence of 
profiles of blowup solutions.

%-----------------------------------------------------------------------
\section{Preliminaries}
\label{sec-pre}
In the rest of this paper, we 
consider radial solutions only, although some of the results in this
preliminary section, notably those concerning the energy functional,
hold for nonradial solutions. 
Notice that radial solutions of \eqref{eq-entire} or \eqref{eq-ancient},
viewed as functions of $r$ and $t$,  satisfy the equation
\begin{equation} \label{eq-Fuj-rad}
 u_t=u_{rr}+\frac{N-1}r u_r+u^p
\quad\hbox{ in }\ (0,\infty)\times(-\infty,T) 
\end{equation}
with $T\leq\infty$,
and the rescaled functions $v=v(\rho,s)$ 
(where $\rho:=|y|$) 
satisfy the equation
\begin{equation} \label{eq-v-rad}
v_s=v_{\rho\rho}+\frac{N-1}\rho v_\rho-\frac \rho2 v_\rho-\frac v{p-1}+v^p
\quad\hbox{ in }\ (0,\infty)\times(-\infty,\infty).
\end{equation}

\subsection{Universal estimates}
\label{subsec-ue}
The following universal estimates
for positive radial solutions $u$ of \eqref{eq-entire},
\eqref{eq-ancient} and the corresponding
rescaled functions $v$ will play an important role in our analysis.
Notice first that if $v$ is any solution of \eqref{eq-v-rad}
and $u$ is defined by \eqref{vu}, then $u$ is a solution
of \eqref{eq-Fuj-rad}, hence any solution $v$ of \eqref{eq-v-rad}
corresponds to a solution $u$ of  \eqref{eq-Fuj-rad}.

%-------------------
\begin{proposition} \label{prop-ub}
Assume $p>1$. Then
there exists $C=C(N,p)>0$ with the following properties:
If $u=u(r,t)$ is a positive solution
of \eqref{eq-Fuj-rad} in $Q_T:=(0,\infty)\times(-\infty,T)$ 
with $T\leq\infty$, 
then
\begin{equation} \label{est-u}
u(r,t)+|u_r(r,t)|^{2/(p+1)}+|u_{rr}(r,t)|^{1/p}
\leq C(r^{-2/(p-1)}+m(t))
\ \hbox{ in }\ Q_{T},
\end{equation}
where 
$m(t)=(T-t)^{-1/(p-1)}$ if $T<\infty$ and $m(t)=0$ if $T=\infty$.
If $T<\infty$, then the corresponding rescaled function $v=v(\rho,s)$
satisfies
\begin{equation} \label{est-v}
v(\rho,s)+|v_\rho(\rho,s)|^{2/(p+1)}+|v_{\rho\rho}(\rho,s)|^{1/p}
\leq C(\rho^{-2/(p-1)}+C_T)
\ \hbox{ in }\ Q_\infty,
\end{equation}
where $C_T=1$.
If $u$ is an entire solution and 
$v$ is defined by \eqref{vu} with $T<\infty$,
then \eqref{est-v} is true with $C_T=0$.
\end{proposition}

\begin{proof}
The proof is a straightforward modification
of the doubling and rescaling arguments
in \cite{PQS07} and the Liouville theorem
for positive solutions of \eqref{eq-entire} with $N=1$;
cp.~also \cite{BPQ11}.
First notice that \eqref{est-u} and \eqref{vu}
imply \eqref{est-v}, hence it is sufficient
to prove \eqref{est-u}.
In addition, \eqref{est-u} with $T=\infty$
is a consequence of \eqref{est-u} with $T<\infty$
since the constant $C$ does not depend on $T$.
Consequently, we may assume $T<\infty$.

Set
$$ M[u](r,t):=u(r,t)^{(p-1)/2}+|u_r(r,t)|^{(p-1)/(p+1)}+|u_{rr}(r,t)|^{(p-1)/2p}. $$
Assume that \eqref{est-u} is not true.
Then there exist $T_k$,
solutions $u_k$ of \eqref{eq-Fuj-rad} in $Q_{T_k}$ and
points $(r_k,t_k)\in Q_{T_k}$ such that
\begin{equation} \label{Mk}
M_k:=M[u_k](r_k,t_k) > 2k/d_k(r_k,t_k), \quad k=1,2,\dots,
\end{equation}
where $d_k(r,t):=\min(r,\sqrt{T_k-t})$ denotes the parabolic
distance of $(r,t)$ to the topological boundary of $Q_{T_k}$.
Then \cite[Lemma~5.1]{PQS07a} guarantees that 
after possible modification of $(r_k,t_k)$,
\eqref{Mk} holds and, in addition, 
we may
assume $M[u_k](r,t)\leq 2M_k$ whenever $|r-r_k|+\sqrt{|t-t_k|}<k/M_k$.
Set 
$$U_k(\rho,s):=\lambda_k^{2/(p-1)}u_k(r_k+\lambda_k\rho,t_k+\lambda_k^2s),$$
where $\lambda_k:=1/M_k$.
Then $U_k$ satisfies the equation
$$  U_s = U_{\rho\rho}+\frac{N-1}{r_k/\lambda_k+\rho}U_\rho+U^p, $$
$U_k,(U_k)_\rho,(U_k)_{\rho\rho}$ are bounded in $\{(\rho,s):|\rho|+\sqrt{|s|}<k\}$
by a constant independent of $k$, and
$U_k(0,0)+|(U_k)_\rho(0,0)|+|(U_k)_{\rho\rho}(0,0)|\geq c_0>0$.
Clearly, $r_k/\lambda_k\to\infty$.
Using standard parabolic estimates, we conclude that
(a suitable subsequence of)
$\{U_k\}$ converges to a positive solution of \eqref{eq-entire} with $N=1$.
But this contradicts the corresponding Liouville theorem, 
see \cite{PQS07}, for example.
\end{proof}

%-------------------------------------------------------------------

\subsection{Lyapunov functional}\label{subs:energy}

Equation \eqref{eq-v} can also be written
in the form
\begin{equation} \label{eq-v2}
 v_s=\frac1\varrho \nabla\cdot(\varrho\nabla v)-\frac v{p-1}+v^p,
\qquad y\in\R^N,\ s\in(-\infty,\infty),
\end{equation}
where $\varrho$ is the Gaussian weight defined by 
$$\varrho(y):=e^{-|y|^2/4}.$$
It is known that this problem possesses the Lyapunov functional
$$E(w)=\int_{\R^N}\Bigl(\frac12|\nabla w|^2+\frac1{2(p-1)} w^2
       -\frac1{p+1}w^{p+1}\Bigr)\varrho\,dy. $$
More precisely, we have the following proposition
(see \cite[Proposition~23.8]{QS07} for more details;
note that the assumption $v(\cdot,s_0)\in BC^1(\R^N)$ in \cite{QS07}
is satisfied for radial solutions of \eqref{eq-v2}
due to Proposition~\ref{prop-ub}
and the fact that we consider classical solutions).

\begin{proposition} \label{prop-energy}
Let $p>1$ and let $v$ be a positive radial solution of \eqref{eq-v2}.
Then $E(v(\cdot,s))\geq 0$ and
\begin{equation} \label{dE}
 \frac{d}{ds}E(v(\cdot,s)) = -\int_{\R^N}v_s^2(y,s)\varrho(y)\,dy,
\end{equation}
for all $s\in\R$.
\end{proposition}

Notice also that
\begin{equation} \label{Eofw}
E(w)=\Bigl(\frac12-\frac1{p+1}\Bigr)\int w^{p+1}\varrho\,dy>0
\end{equation}
for any bounded positive radial steady state $w$ of \eqref{eq-v2} (or \eqref{eq-v})
and this also remains true for the singular steady state
$\phi_\infty$ if $p>p_S$ since 
$\phi_\infty\in H^1_{loc}\cap L^{p+1}_{loc}(\R^N)$
for such $p$.

It is known that if $p>p_S$ and $w$ is a positive radial non-constant steady state
of \eqref{eq-v2} or $w=\phi_\infty$, then $E(w)>E(\kappa)$,
see \cite[Remark 1.17]{MM11}.
In particular,
\begin{equation} \label{Einftykappa}
E(\phi_\infty)>E(\kappa).
\end{equation}
The proof of \eqref{Einftykappa} in \cite{MM11} is quite long and involved.
In the proof of the following proposition we use a simpler and more
direct argument 
to prove \eqref{Einftykappa}
(cf. also the beginning of Subsection~3.3 in \cite{MM11}).
This argument enables us also to show that the ratio $E(\phi_\infty)/E(\kappa)$
is monotone with respect to $p$.

\begin{proposition} \label{propEinftykappa}
  Let $N>2$ and $F:(p_S,\infty)\to\R$
  denote the function $p\mapsto E(\phi_\infty)/E(\kappa)$.
Then $F$ is decreasing, $\lim\limits_{p\to p_S}F(p)=\infty$ and\/
$\lim\limits_{p\to\infty}F(p)=1$.
\end{proposition}

\begin{proof}
Set $\xi:=(p+1)/(p-1)$. Then a straightforward calculation based on
\eqref{Eofw} shows $F(p)=f(\xi)$, where
$$f:(1,N/2)\to\R:\xi\mapsto \frac{\Gamma(N/2-\xi)}{\Gamma(N/2)}
\Bigl(\frac{N-(1+\xi)}2\Bigr)^\xi,$$
and $\Gamma$ stands for the standard gamma function. 
Since $\lim_{\xi\to1}f(\xi)=1$ and $\lim_{\xi\to N/2}f(\xi)=\infty$,
it is sufficient to prove $f'(\xi)>0$ for $\xi\in(1,N/2)$.
This inequality is equivalent to
\begin{equation} \label{ineq_psi}
 \psi\Bigl(\frac N2-\xi\Bigr)<\log\frac{N-(1+\xi)}2-\frac \xi{N-(1+\xi)},
\end{equation}
where 
$$ \psi(z):=\frac{\Gamma'(z)}{\Gamma(z)}<\log z-\frac1{2z} 
  \quad\hbox{ for }\ z>0,$$
see \cite[6.3.21]{AS}.
Consequently, to prove \eqref{ineq_psi} it is sufficient to show
$$ \log\frac{N-2\xi}2-\frac1{N-2\xi}
\leq  \log\frac{N-(1+\xi)}2-\frac \xi{N-(1+\xi)}, $$
which is equivalent to
$$ \log\Bigl(1+\frac{\xi-1}{N-2\xi}\Bigr)\geq \frac \xi{N-(1+\xi)}-\frac1{N-2\xi}.$$
Setting $\eta:=(\xi-1)/(N-2\xi)$, the last inequality is equivalent to
$$ \log(1+\eta) \geq \eta-\frac \eta{\eta+1}\frac{N\eta+1}{N-2}\qquad (\eta>0).$$
Using the estimate $(N\eta+1)/(N-2)\geq \eta$ we see that it is sufficient to show
$$\log(1+\eta)\geq \frac \eta{\eta+1}\qquad (\eta>0).$$
The last inequality is easy to prove (consider the derivatives 
of the left and right hand sides, for example).
\end{proof}

%-------------------------------------------
\subsection{Zero number}
Recall that radial solutions of \eqref{eq-entire} or \eqref{eq-ancient}
satisfy equation \eqref{eq-Fuj-rad} with $T\leq\infty$,
and the boundary condition $u_r(0,t)=0$,
and the rescaled functions $v=v(\rho,s)$ 
satisfy equation \eqref{eq-v-rad}
and the boundary condition $v_\rho(0,s)=0$.
The singular steady state $\phi_\infty=\phi_\infty(r)$ satisfies
both \eqref{eq-Fuj-rad} and \eqref{eq-v-rad}
and the boundary condition $\phi_\infty(0)=\infty$.

If $u_1,u_2$ are radial solutions of \eqref{eq-entire} or \eqref{eq-ancient}
(or $u_1,u_2$ are radial solutions of \eqref{eq-v}),
then $U:=u_1-u_2$ solves the linear equation
\begin{equation} \label{eq-diff}
U_t =U_{rr}+\frac{N-1}rU_r-c\frac r2 U_r+ fU
\quad\hbox{ in }\ (0,\infty)\times(-\infty,T)
\end{equation}
and satisfies the boundary condition $U_r(0,t)=0$, 
where $T\leq\infty$,
$c\in\{0,1\}$ and 
$f=f(r,t)$ is in $L^\infty((0,\infty)\times(t_1,t_2))$
whenever $-\infty<t_1<t_2<T$
(the boundedness comes from Proposition~\ref{prop-ub}
and the fact that we consider classical solutions).
If $u_1\equiv\phi_\infty$ and $u_2$ is as above,
then $U$ satisfies \eqref{eq-diff}, 
the boundary condition $U(0,t)=\infty$,
and $f\in L^\infty((\delta,\infty)\times(t_1,t_2))$
for any $\delta>0$.

If $I\subset[0,\infty)$ is an interval and
$g:I\to\R$ is a continuous function, we
denote by $z_I(g)$ the number of zeros of $g$ in $I$. 
We also set $z(g)=z_{(0,\infty)}g$.

The next proposition follows from zero number theorems of
\cite{CP:96, Ma99}.

\begin{proposition} \label{prop-zero}
Let $U$ be as above, $U\not\equiv0$, $t_1<t_2<T$. Then we have:

\strut\hbox{\rm(i)}
The function $t\mapsto z(U(\cdot,t))$ is nonincreasing.
If $z(U(\cdot,t_1))<\infty$ and
\begin{equation} \label{zero-ass}
U(r_0,t_0)=U_r(r_0,t_0)=0 \hbox{ for some }
 r_0\geq0 \hbox{ and }t_0\in(t_1,t_2),
\end{equation}
then 
\begin{equation} \label{zero-concl}
z(U(\cdot,t))>z(U(\cdot,s)) \hbox{ for all }t_1<t<t_0<s<t_2.
\end{equation}

\strut\hbox{\rm(ii)}
Assume $R>0$, $U(R,t)\ne0$ for all $t\in[t_1,t_2]$.
Then the function $t\mapsto z_{(0,R)}(U(\cdot,t))$
is nonincreasing and finite. If \eqref{zero-ass} is true for some $r_0\in[0,R)$,
then  \eqref{zero-concl} is true with $z$ replaced by $z_{(0,R)}$.
\end{proposition}

%----------------------------------------------------

\subsection{Steady states and limit sets of \eqref{eq-v-rad}}
\label{subsec-ss}

In what follows  we  assume that $v$ is a positive solution of \eqref{eq-v-rad}
and $p>p_S$.  
Estimate \eqref{est-v} guarantees that the Lyapunov functional
$E(v(\cdot,s))$ is uniformly bounded for $s\in\R$
and $E(v(\cdot,t_k))\to E(w)$ whenever 
$v(\cdot,t_k)\to w\hbox{ in }C^1_{loc}(0,\infty)$.
Consequently, standard arguments (see~Appendix G in \cite{QS07}, for example)
show that the $\alpha$- and $\omega$-limit sets
$$ \alpha(v):=\{w\in C^1(0,\infty):(\exists t_k\to-\infty)\
 v(\cdot,t_k)\to w\hbox{ in }C^1_{loc}(0,\infty)\}, $$
$$ \omega(v):=\{w\in C^1(0,\infty):(\exists t_k\to\infty)\
 v(\cdot,t_k)\to w\hbox{ in }C^1_{loc}(0,\infty)\}, $$
are nonempty connected sets
consisting of nonnegative steady states
of \eqref{eq-v-rad}. 
In addition, if $v$ corresponds to an entire solution $u$
(hence \eqref{est-v} is true with $C_T=0$)
and $v$ is bounded in $(0,\infty)\times(T_1,T_2)$
for some $-\infty\leq T_1<T_2\leq\infty$,
then the convergence $v(\cdot,t_k)\to w$ in $C^1_{loc}(0,\infty)$
with $t_k\in(T_1,T_2)$  implies the convergence 
$v(\cdot,t_k)\to w$ in $BC^1(0,\infty)$.

We now summarize further useful  properties of $\alpha(v)$
and $\omega(v)$ reflecting the structure of steady states
of the present problem. In particular, we show that $\alpha(v)$
and $\omega(v)$ are singletons.

First note that estimate
\eqref{est-v} with $C_T=0$ implies $\kappa\notin\alpha(v)$.
Our assumption $p>p_S$ guarantees that
$\phi_\infty$ is the only nonnegative 
steady state of \eqref{eq-v-rad} satisfying $\limsup_{\rho\to0}w(\rho)=\infty$,
see  \cite[Theorem~1.2]{M10} or \cite{Q18}.
Notice also that $0\notin\alpha(v)$ 
since $E(0)=0<E(w)$ for any positive steady state of \eqref{eq-v-rad}
(cp. \eqref{Eofw})
and $s\mapsto E(v(\cdot,s))$ is decreasing unless $v$ is  
a steady state.

Any nonnegative steady state $w$ of \eqref{eq-v-rad} 
satisfying $\limsup_{\rho\to0}w(\rho)<\infty$
is uniquely determined by its value at $\rho=0$.
If $w$ is nonconstant, then \cite[Lemmas~2.2--2.3]{M09} and \cite{BE88}
yield the following relations
\begin{equation}
  \label{eq:2}
  w(0)>\kappa, \qquad w'<0\ \hbox{ on }\ (0,\infty), \qquad
z(w-\phi_\infty)\geq2.
\end{equation}
Denote by ${\cal A}$ the set of $a\in[0,\infty)$ 
for which there exists a steady state
$w_a\geq0$ of \eqref{eq-v-rad} satisfying $w_a(0)=a$.
By \cite[Proposition~2.3 and the proof of Lemma~2.4]{Ma99},
for any $a\in{\cal A}\setminus\{0,\kappa\}$ there exists 
$c_a:=\lim_{\rho\to\infty}w_a(\rho)\rho^{2/(p-1)}\in(0,\infty)\setminus\{L\}$,
and the mapping $a\mapsto c_a:{\cal A}\setminus\{0,\kappa\}\to(0,\infty)$ is
injective.
In particular, $z(w_a-\phi_\infty)<\infty$ for any $a\in{\cal A}$.
Set 
$${\cal A}_k:=\{a\in{\cal A}:z(w_a-\phi_\infty)=k\}, \qquad
k=0,1,2,\dots.$$ 
By \eqref{eq:2},  ${\cal A}_0=\{0\}$ and ${\cal A}_1=\{\kappa\}$.
As proved in   \cite{PQ19}, the set ${\cal A}$ is discrete.
This---in conjunction with the uniqueness of the unbounded
positive steady state $\phi_\infty$---shows that
for any positive solution $v$ of \eqref{eq-v-rad},
the sets  $\alpha(v)$ and $\omega(v)$
are singletons consisting of either $\phi_\infty$ or $w_a$ for some
$a\in{\cal A}$. 

As already mentioned in the introduction,
if $p>p_L$, then ${\cal A}=\{0,\kappa\}$, i.e.\ $w_0\equiv0$ and
$w_\kappa\equiv\kappa$ are the only bounded
nonnegative steady states of \eqref{eq-v-rad}.
In this case, each of the sets $\alpha(v)$ and $\omega(v)$
has to be one of the sets $\{\phi_\infty\}$, $\{\kappa\}$, or $\{0\}$.
We also know that
$\alpha(v)\ne\{0\}$ (and $\alpha(v)\ne\{\kappa\}$ if
$v$ corresponds to an entire solution $u$).
Proposition~\ref{propEinftykappa} guarantees
$\omega(v)\ne\{\phi_\infty\}$.

Let now $p_S<p<p_{JL}$. Then each of the sets ${\cal A}_k$ is nonempty 
(see \cite{T87,L88,BQ89,FP09,NS19} and references therein)
and bounded (this follows from the first sentence in the proof of
\cite[Lemma 2.2]{FM07}, for example), hence finite.
On the other hand, an easy contradiction argument shows
$\inf{\cal A}_k\to\infty$ as $k\to\infty$.

The arguments in the proof of \cite[Proposition 2.4]{FP09}
show that if  $w_1,w_2$ are two different positive steady states of \eqref{eq-v-rad}
(possibly unbounded), then
\begin{equation} \label{eq-a} 
 w_1(\rho)=w_2(\rho) \hbox{ for some }\rho>0 
  \ \hbox{ implies }\ w_1(\rho)\geq\kappa.
\end{equation} 
Hence,
$w_1$ and $w_2$ do not intersect for large values of $\rho$.
This is also a consequence of
Proposition \ref{prop-comp} below,
where we examine similar intersection properties
for time-dependent solutions of \eqref{eq-v-rad}. 

%----------------------------------------------------
\subsection{Comparison arguments and
intersections of solutions of \eqref{eq-v-rad} for large $\rho$}
\label{subsec-intersect}

Let $v_1,v_2$ be two positive solutions of \eqref{eq-v-rad}.
Then $V:=v_1-v_2$ satisfies
\begin{equation} \label{eq-v1v2}
V_s=V_{\rho\rho}+\frac{N-1}rV_\rho-\frac \rho2 V_\rho+ fV,
\end{equation}
where 
$$f=f(\rho,s)=-\frac1{p-1}+\begin{cases}
pv_1^{p-1} &\hbox{ if }V(\rho,s)=0,\\
\frac{v_1^p-v_2^p}{V} &\hbox{ otherwise.}
\end{cases}$$
By the Mean Value Theorem,
\begin{equation} \label{est-f}
v_1,v_2\leq C_v \quad\Rightarrow\quad f\leq C_0:=-\frac1{p-1}+pC_v^{p-1}.
\end{equation}
In particular, 
\begin{equation} \label{est-f2}
f\leq-\delta_0:=-\frac1{2(p-1)} \quad\hbox{provided}\quad 
 C_v\leq c_0:=\Bigl(\frac1{2p(p-1)}\Bigr)^{1/(p-1)}.
\end{equation}

\begin{proposition} \label{prop-comp}
Let $v_1,v_2,V,\allowbreak c_0,\delta_0$ be as above, $s_0\in\R$,
$\rho_0>0$ and $V(\rho_0,s_0)\neq0$.
Set
\begin{equation} \label{D}
\left.\begin{aligned}
D &:=\{(\rho,s)\in(0,\infty)\times(-\infty,s_0]:V(\rho,s)\ne0\}, \\
D_0 &:=\hbox{ the connected component of $D$ containing $(\rho_0,s_0)$}, \\
\Omega(s) &:=\{\rho:(\rho,s)\in D_0\}. 
\end{aligned}
\quad \right\}
\end{equation}
Assume
\begin{equation}
  \label{eq:4}
  v_2\leq c_0 \quad\hbox{in }D_0,
\end{equation}
and 
\begin{equation} \label{v1v2-infty}
\lim_{\rho\to\infty,\ \rho\in\Omega(s)}v_2(\rho,s)=0,
\quad\hbox{locally uniformly in $s$}. 
\end{equation}
Then $V(\rho_0,s_0)>0$.
\end{proposition}

In applications of this proposition, we verify condition 
\eqref{v1v2-infty} using an a priori bound, such as 
\eqref{est-v} with $C_T=0$. 
By the same a priori bound,
we will have \eqref{eq:4}
verified, provided 
$\rho_1(s):= \inf\Omega(s)$ is large enough for all $s$.

Notice that if $v_1$ 
also satisfies such an a priori bound,
then $v_1$ and $v_2$ can be interchanged. In this case,
Proposition \ref{prop-comp}
says in effect that   $\rho_1(s)$ cannot be 
large for all $s$. This in particular entails
statement \eqref{eq-a} for steady states, as noted at the end of the
previous subsection.  

\begin{proof}[Proof of Proposition \ref{prop-comp}] Let
\begin{equation} \label{D1}
\left.\begin{aligned}
m(s) &:=\sup_{\Omega(s)}|V(\cdot,s)|, \\
S &:=\inf\{s<s_0:\Omega(s)\ne\emptyset\}, 
\end{aligned}
\quad \right\}
\end{equation}
The proof is by contradiction. Assume that $V(\rho_0,s_0)<0$.
Then 
\begin{equation} \label{v1v2-small}
0<v_1<v_2\leq c_0 \quad\hbox{in }D_0,
\end{equation}
hence $m(s)\leq c_0$ for all $s\leq s_0$.
The comparison principle used for
equation \eqref{eq-v1v2} together with estimate \eqref{est-f2}
give
\begin{equation} \label{est-m}
m(s_0)\leq e^{-\delta_0(s_0-s)}m(s)\quad\hbox{for}\quad s\in(S,s_0).
\end{equation}
If $S=-\infty$, then \eqref{est-m} and \eqref{v1v2-small}
yield $m(s_0)=0$.
If $S>-\infty$, then \eqref{v1v2-infty}
and the continuity of $V$
guarantee $m(s)\to0$ as $s\to S+$,
hence $m(s_0)=0$ again.
But $m(s_0)=0$ contradicts our assumption $V(\rho_0,s_0)\ne0$.
\end{proof}

%----------------------------------------------------
\section{Proof of Theorem~\ref{thmL1}}
\label{sec-proofs}
The proof of Theorem~\ref{thmL1} is long and rather technical at places.
We first give an outline.
Let $u=u(r,t)$ be a positive solution of \eqref{eq-entire} with
$p>p_L$. Fixing any $T\in \R$,
let $v$ be the corresponding rescaled solution
of \eqref{eq-v-rad}.
Using considerations in  Subsection~\ref{subsec-ss}, we first show
easily that $\alpha(v)=\{\phi_\infty\}$. Thus, formally, $v$ can be
viewed as a solution on the unstable manifold of the singular steady
state. (The term ``manifold'' is used loosely here; the manifold
structure of the solutions approaching $\phi_\infty$ backward in time
is not actually established.) 
At the same time, as observed in \cite{PY05},
the solutions of \eqref{eq-v-rad} corresponding to the radial
steady states of the original equation
\eqref{eq-entire} form a one-dimensional manifold that can be
considered as the principal part of the unstable manifold of $\phi_\infty$:
As time approaches $-\infty$,
these rescaled solutions approach
$\phi_\infty$ monotonically and at an exponential 
rate given by the principal eigenvalue of
the linearization of the right-hand side of \eqref{eq-v-rad} at
$\phi_\infty$.  
Our main goal is to derive suitable estimates on
$\phi_\infty-v$ in order to show that the entire solution
$v$ has to lie on the principal part of the unstable manifold,
or, in other words, $u$ is a steady state. This is achieved
by careful analysis of the
abstract form of equation \eqref{eq-v-rad} and, in particular,
of the remainder on the right-hand side after 
the linearization has been subtracted from it. 
This analysis, which is really the  crux of our proof,
is carried out in the next subsection. 
We remark that the proof of Theorem \ref{thmLFPY}(ii), as given in
\cite{PY05}, follows a similar general scenario. However, the bounds
$\phi_\alpha\leq u(\cdot,t)\leq \phi_\infty$ assumed there
make all the necessary estimates considerably simpler, even when
nonradial solutions are allowed; those estimates from \cite{PY05}
are of little help in our present analysis
(we make use of other technical results from \cite{PY05}). 

Another ingredient of the proof of Theorem \ref{thmL1} is the radial
monotonicity of the entire solutions, which we prove in Subsection
\ref{subsec-monot} for any $p>p_S$. 
We then  complete the proof of the theorem in Subsection
\ref{subsec-completion}.

\subsection{Linearization of \eqref{eq-v-rad} at $\phi_\infty$ and estimates of the remainder}
\label{subsec-Ah}

In this subsection,  we first assume
assume $p>p_{JL}$ (some abstract results that we recall are valid in
this range), and then focus on the case $p>p_L$.

Set  $a(\rho):=\rho^{N-1}e^{-\rho^2/4}$. 
We consider the weighted Lebesgue space $X:=L^2(0,\infty;a(\rho)d\rho)$
endowed with the scalar product
$$ \langle f,g\rangle:=\int_0^\infty f(\rho)g(\rho)a(\rho)\,d\rho $$
and the corresponding norm $\|f\|_X=\langle f,f\rangle^{1/2}$.
Let
$$Y:=\{f\in H^1_{loc}(0,\infty): f,f'\in X\}$$
be endowed with the norm $\|f\|_Y:=\|f\|_X+\|f'\|_X$.
It was shown in \cite[Lemma 2.3]{HV94} that the operator
\begin{equation} \label{A}
Af:=f''+\Bigl(\frac{N-1}\rho-\frac \rho2\Bigr)f'
       +\Bigl(\frac{pL^{p-1}}{\rho^2}-\frac1{p-1}\Bigr)f 
\end{equation}
with domain
$$ D(A):=\{f\in Y: Af\in X \hbox{ in  the distributional sense}\}$$
can be extended in a unique way to a self-adjoint operator in $X$
(still denoted by $A$), with the following properties:
\begin{itemize} 
\item[(A1)] $D(A)\subset Y$,
\item[(A2)] $(\exists c_A>1)(\forall \phi\in D(A))\ \langle \phi,A\phi\rangle\leq(c_A-1)\langle \phi,\phi\rangle$,
\item[(A3)] the spectrum $\sigma(A)$ consists of a sequence of simple eigenvalues
$$ \mu_j:=-\Bigl(\frac\beta2+\frac1{p-1}+j\Bigr), \quad j=0,1,2,\dots,$$
where
$$ \beta:= \frac12\bigl(-(N-2)+\sqrt{(N-2)^2-4pL^{p-1}}\bigr)<0,$$
and the corresponding eigenfunctions (normalized in $X$) have the form 
$\vartheta_j(\rho)=\hat c_j \rho^\beta M_j(\rho^2/4)$,
where $\hat c_j>0$,
$$ M_j(z):= M\Bigl(-j,\beta+\frac N2,z\Bigr) $$
and $M$ denotes the standard Kummer function (hence $M_j$ is a
polynomial of degree $j$).
Also, for $j=0,1,\dots$, the function
$\vartheta_j$ has exactly $j$ zeros, all of them positive and simple. 
\end{itemize}

The operator $-\tilde A:=-A+c_A$ 
is a positive self-adjoint operator
and its fractional powers $(-\tilde A)^\alpha$ are well defined for all $\alpha\in\R$
(see \cite[Section III.4.6]{A95}).
We  denote by $\{(X_\alpha,-A_\alpha):\alpha\in[-1,1]\}$
the corresponding fractional interpolation-extrapolation scale
of spaces and operators (see \cite[Section V.1]{A95} for its definition
and properties); the norm in $X_\alpha$ will be denoted by $\|\cdot\|_\alpha$.
In particular, $X_0=X$, $A_1=\tilde A$, $X_1=D(\tilde A)$,
$X_{-1}=X_1'$ (where the duality is taken with respect to the duality pairing $\langle\cdot,\cdot\rangle$).
Recall also that this scale is equivalent to the scale generated by $(X,-\tilde A)$
and the complex interpolation functor $[\cdot,\cdot]_\theta$. 
The space $X_{1/2}$ is isomorphic to $Y$, see \cite[Lemma 2.4]{HV94}.
By general result of \cite[Section V.2]{A95}, 
$A_\alpha$ generates an analytic semigroup 
$e^{sA_\alpha}$ in $X_\alpha$ and the following estimate is true for any $\sigma\geq0$
\begin{equation} \label{est-semigroup}
 \|e^{s(A_\gamma+\sigma)}\phi\|_\alpha\leq cs^{\gamma-\alpha}e^{\sigma s}\|\phi\|_\gamma,\qquad
 -1\leq\gamma\leq\alpha\leq1,\ \ s>0.
\end{equation}

If $v$ is as in Proposition~\ref{prop-ub}, $w:=\phi_\infty-v$, 
$$h:=\phi_\infty^p-v^p-p\phi_\infty^{p-1}w,$$
and $f:=v(\cdot,s)$ for some $s$, then
estimate \eqref{est-v} and formulas \cite[(2.52), (2.59)]{HV94} show
that \eqref{A} and the variation-of-constants formula
$$ w(s)=e^{(s-s_0)A}w(s_0)+\int_{s_0}^s e^{(s-\tau)A}h(\cdot,\tau)\,d\tau $$
are true with $A$ replaced by $A_{-1/2}+c_A$.
Since no confusion seems likely, in what follows we set
$A:=A_{-1/2}+c_A$. In particular, estimate \eqref{est-semigroup} implies
\begin{equation} \label{est-semigroupA}
 \|e^{sA}\phi\|_\alpha\leq cs^{\gamma-\alpha}e^{c_A s}\|\phi\|_\gamma,\qquad
 -1/2\leq\gamma\leq\alpha\leq1,\ \ s>0.
\end{equation}

Henceforth we assume that $p>p_L$. 

Let $v,w,h,A,\mu_j,\vartheta_j,(X_\alpha,\|\cdot\|_\alpha)$ be as
above. Crucial for our proof of Theorem \ref{thmL1} is 
a good understanding of the behavior of $v$ in the following case: 
\begin{equation} \label{ass-v}
0\leq v(\cdot,s)<\phi_\infty \qquad\hbox{and}\qquad
\alpha(v)=\{\phi_\infty\}.
\end{equation}
Here, the $\alpha$-limit set $\alpha(v)$ is as in
Subsection~\ref{subsec-ss}. 
In the following proposition we prove,  
loosely speaking, that along a sequence of times the function 
$v$ approaches $\phi_\infty$ in the direction of the eigenfunction
$\vartheta_0$ and at the rate $\exp(\mu_0 s)$.
\begin{proposition} \label{prop-est}
Under the above assumptions and notation, 
there exist a constant $c>0$ and a sequence $s_k\to-\infty$ such that
\begin{equation} \label{final-est3}
\|w(\cdot,s_k)-ce^{\mu_0s_k}\vartheta_0\|_0=o(e^{\mu_0s_k})\quad\hbox{as}\quad k\to\infty.
\end{equation}
\end{proposition}

\begin{proof} Recall that $p>p_L$ implies that
 $\mu_0>\mu_1>0>\mu_2$ (this can be easily checked using the formulas
 in (A3)).

Let $P$ be the orthogonal projection 
onto the orthogonal complement of $\{\vartheta_0,\vartheta_1\}$ in $X$.
Let $\xi_0,\xi_1$ be defined by
\begin{equation} \label{w-decomp}
 w(\cdot,s)=\xi_0(s)\vartheta_0+\xi_1(s)\vartheta_1+\tilde w(\cdot,s),
\end{equation}
where $\tilde w(\cdot,s):=P(w(\cdot,s))$. 
Since $E(v(\cdot,s))\to E(\phi_\infty)$ as $s\to-\infty$ (see Section
\ref{subsec-ss}), 
we have $w(\cdot,s)\to 0$ in $X$, hence
\begin{equation} \label{wto0}
 \|\tilde w(\cdot,s)\|_0\to0\quad\hbox{and}\quad \xi_i(s)\to0 \ (i=0,1) 
 \quad\hbox{as}\quad s\to-\infty.
\end{equation}

Our first goal is to prove that 
\begin{equation} \label{est-w0}
\|\tilde w(s)\|_0=o(|\xi_0(s)|+|\xi_1(s)|)\ \text{ as \ $s\to-\infty$.}
\end{equation}
We start with some estimates of the function $h$. 

By assumption, $h\leq0<w$, hence $\langle h,w\rangle\leq 0$.
In addition, 
$$0\leq -h=p(\phi_\infty^{p-1}-v_\theta^{p-1})w
 \leq p(\phi_\infty^{p-1}-v^{p-1})w
\leq p(p-1)\phi_\infty^{p-2}w^2 $$
for some $v_\theta\in(v,\phi_\infty)$, hence, 
given any $\delta\in[0,1)$,
\begin{equation} \label{est-h1}
 -h\leq f_\delta w^{1+\delta},
\end{equation}
where $f_\delta=f_\delta(\rho,s)$ is given by
$$ f_\delta:=
 [p(\phi_\infty^{p-1}-v^{p-1})]^{1-\delta}[p(p-1)\phi_\infty^{p-2}]^{\delta}.$$
Note that
\begin{equation} \label{est-h2}
\begin{aligned}
&f_\delta(\rho,s)\to 0 \ \hbox{ as }\ s\to-\infty,\\
&|f_\delta(\rho,s)|\leq C\rho^{-\nu_\delta},\ \hbox{ where }\
\nu_\delta:=2(1-\delta)+2\frac{p-2}{p-1}\delta\leq2.
\end{aligned}
\end{equation}

Choose $\zeta\in(1-\beta/2-N/4,1/2)$, $\zeta\geq0$.
We will specifically take $\zeta=0$ when 
 $1-\beta/2-N/4<0$, which is the case
if 
$$p>p_H:=1+4\frac{N+2\sqrt{N}-4}{N^2-12N+16}.$$
Clearly, there is $z>1$ such that 
$$\frac Nz>\frac N2-2\zeta \quad\hbox{and}\quad \frac N{z'}>2-\beta,$$
where $z':=z/(z-1)$.
Fixing such $z$, if $\delta_0>0$ is  small enough, we have 
\begin{equation} \label{z-cond2}
\frac N{z(1+\delta)}>\frac N2-2\zeta \quad\hbox{and}\quad 
\frac N{z'}>\nu_\delta-\beta>2 \quad\hbox{ for any }\delta\in[0,\delta_0].
\end{equation}
Since $|\vartheta_i(\rho)|\leq C\rho^{\beta}(1+\rho^{2i})$, $i=0,1$
(cp. (A3)), 
estimate
\eqref{est-h1} gives the following relations (omitting the argument
$\rho$ of the indicated functions)
\begin{equation} \label{est-hvartheta} 
 \begin{aligned} 
|h\vartheta_i a| &\leq C f_\delta \rho^\beta(1+\rho^{2i})w^{1+\delta}a \\
 &\leq \begin{cases}
   C\bigl((f_\delta \rho^\beta)^{z'} a\bigr)^{1/z'}\bigl(w^{(1+\delta)z} a\bigr)^{1/z}, & \rho\leq1,\\
   C\bigl((f_\delta\rho^{\beta+2})^{2/(1-\delta)} a\bigr)^{(1-\delta)/2}\bigl(w^2 a\bigr)^{(1+\delta)/2}, & \rho>1.
\end{cases} 
\end{aligned} 
\end{equation}
Next, the embedding inequalities $\|w\|_0\leq C\|w\|_{\zeta}$,
$$\|w|_{(0,1)}\|_{L^{z(1+\delta)}(0,1;a(\rho)d\rho)} \leq
C\|w\|_\zeta$$
(the latter follows from  \eqref{z-cond2}), 
and the H\"older inequality imply
$$
|\langle h,\vartheta_i\rangle| \leq
 C\Bigl[\Bigl(\int_0^1(f_\delta\rho^\beta)^{z'}a\,d\rho\Bigr)^{1/z'}
 +\Bigl(\int_1^\infty(f_\delta\rho^{\beta+2})^{2/(1-\delta)}a\Bigr)^{(1-\delta)/2}\Bigr]
 \|w\|_\zeta^{1+\delta}. $$
Now, using  
\eqref{est-h2}, 
\eqref{z-cond2}, 
and the Lebesgue theorem, we obtain  
\begin{equation} \label{est-hvartheta2}
|\langle h,\vartheta_i\rangle|=o(\|w\|_\zeta^{1+\delta})\hbox{ as } s\to-\infty, \quad (i=0,1).
\end{equation} 
Notice also that $\vartheta_i\in X_1$, hence
\begin{equation} \label{wtildew}
 \|w(\cdot,s)-\tilde w(\cdot,s)\|_\alpha\leq C(|\xi_0(s)|+|\xi_1(s)|),\qquad \alpha\leq1.
\end{equation}

With the above estimate of the function $h$ at hand, we next examine
differential equations for $\xi_1$, $\xi_2$, and $\|\tilde w\|_0$.
Multiplying the equation 
$w_s=Aw+h$ by $\vartheta_ia$, $i=0,1$, and integrating 
over $(0,\infty)$, we obtain
$$ 
 \dot\xi_i =\mu_i\xi_i +\langle h,\vartheta_i\rangle, \quad i=0,1,
$$
hence \eqref{est-hvartheta2} and \eqref{wtildew} imply
\begin{equation} \label{eq-xi12}
\frac{d}{ds}\xi_i^2 = 2\mu_i\xi_i^2+g_i, \quad i=0,1, 
\end{equation}
where 
$$\begin{aligned}
g_i 
 &=o(\|w\|_\zeta^{1+\delta}|\xi_i|) \\
 &=o((\|\tilde w\|_\zeta+|\xi_0|+|\xi_1|)^{1+\delta}|\xi_i|) \\
 &=o((\|\tilde w\|_\zeta^2+\xi_0^2+\xi_1^2)^{1+\delta/2}) \quad\hbox{as }\ s\to-\infty.
\end{aligned}
$$
Similarly, multiplying the equation
$\tilde w_s = A\tilde w + Ph$ by $2\tilde w$ and using
$\langle A\tilde w,\tilde w\rangle\leq -c_0(\|\tilde w\|_{1/2}^2+\|\tilde w\|_0^2)$,
$\langle h,w\rangle\leq0$, and  \eqref{est-hvartheta2} we obtain
\begin{equation} \label{estw2}
\begin{aligned}
 \frac{d}{ds}\|\tilde w\|_0^2 
&= 2\langle A\tilde w,\tilde w\rangle  + 2\langle h,w-\xi_0\vartheta_0-\xi_1\vartheta_1\rangle \\
&\leq -2c_0(\|\tilde w\|_{1/2}^2+\|\tilde w\|_0^2)  + g_2, 
\end{aligned}
\end{equation}
where 
$g_2=o((\|\tilde w\|_\zeta^2+\xi_0^2+\xi_1^2)^{1+\delta/2})$
as $s\to-\infty$. 

We are now ready to complete the proof of \eqref{est-w0}.
Fix any 
$$\eps_0\in(0,\min(1,\mu_1,c_0)/5).$$
Then there exists $s_0$ such that 
\begin{equation} \label{est-gi}
|g_i|\leq \eps_0^3(\|\tilde w\|_{1/2}^2+\xi_0^2+\xi_1^2)
\quad\hbox{for}\quad s\leq s_0,\ \  i=0,1,2.
\end{equation}
Assume for a contradiction that there exists $s_1\leq s_0$ such that
$$\|\tilde w(\cdot,s_1)\|_0\geq 2\eps_0(|\xi_0(s_1)|+|\xi_1(s_1)|).$$
Then \eqref{estw2} and the convergence
$\|\tilde w(\cdot,s)\|_0\to0$ as $s\to-\infty$
 guarantee the existence of $s_2<s_1$ such that
$$ \begin{aligned}
\|\tilde w(\cdot,s)\|_0 &\geq \eps_0(|\xi_0(s)|+|\xi_1(s)|) \ \hbox{ for }\ s\in[s_2,s_1],\\
\|\tilde w(\cdot,s_2)\|_0 &=\eps_0(|\xi_0(s_2)|+|\xi_1(s_2)|).
\end{aligned}$$ 
In addition, setting $\psi:=\xi_0^2+\xi_1^2$,  \eqref{eq-xi12} implies
\begin{equation} \label{eq-xieta}
2\mu_0\psi+g_0+g_1 \geq \frac{d}{ds}\psi \geq 2\mu_1\psi+g_0+g_1.
\end{equation}
Integrating \eqref{estw2} and the second inequality in \eqref{eq-xieta}
over $(s_2,s_1)$ 
we obtain
$$ \begin{aligned}
4\eps_0^2 \psi(s_1) &\leq \|\tilde w(\cdot,s_1)\|_0^2 \\
&\leq \|\tilde w(\cdot,s_2)\|_0^2-2c_0\int_{s_2}^{s_1}\|\tilde w\|_{1/2}^2\,ds + \int_{s_2}^{s_1}g_2\,ds \\
&\leq 2\eps_0^2\psi(s_2)-2c_0\int_{s_2}^{s_1}\|\tilde w\|_{1/2}^2\,ds + \int_{s_2}^{s_1}g_2\,ds \\
&\leq 2\eps_0^2\psi(s_1)-2c_0\int_{s_2}^{s_1}\|\tilde w\|_{1/2}^2\,ds \\
 &\qquad -2\eps_0^2\mu_1\int_{s_2}^{s_1}\psi\,ds+\int_{s_2}^{s_1}(g_2-2\eps_0^2(g_0+g_1))\,ds,
\end{aligned}$$
and \eqref{est-gi} yields a contradiction.
Thus, \eqref{est-w0} is proved.

We now complete the proof of Proposition \ref{prop-est}, first in the
case $p>p_H$, then in the case $p\in (p_L,p_H]$. 

Assume $p>p_H$ (notice that this assumption is automatically satisfied
if $N\geq16$ due to $p>p_L$). As noted above, in this case
$\zeta:=0$ is our (legitimate) choice. Set
$$\xi:=\xi_0^2, \qquad \eta:=\xi_1^2+\|\tilde w\|_0^2.$$
Then \eqref{eq-xi12} and \eqref{estw2} imply
$$\begin{aligned}
 \dot\xi &= 2\mu_0\xi+o((\xi+\eta)^{1+\delta/2}), \\
 \dot\eta &\leq 2\mu_1\eta+o((\xi+\eta)^{1+\delta/2}).
\end{aligned}$$

Such differential inequalities are considered in 
\cite{PY05}. According to  
\cite[Proposition 4.4(i)]{PY05}, as $s\to-\infty$, we have either 
$\eta(s)=o(\xi(s))$ or
\begin{equation} \label{eta-large}
\xi(s)=o(\eta(s)) \quad\hbox{as}\quad s\to-\infty.
\end{equation}
Assume that \eqref{eta-large} is true.
Then \eqref{est-w0} implies
\begin{equation} \label{xi1-large}
 \|\xi_0(s)\vartheta_0+\tilde w(\cdot,s)\|_0 
 = o(\|\xi_1(s)\vartheta_1\|_0) \quad\hbox{as}\quad s\to-\infty.
\end{equation}
Since $\vartheta_1(r)$ 
changes sign, \eqref{xi1-large} guarantees that 
$w(\cdot,s)=\xi_0(s)\vartheta_0+\xi_1(s)\vartheta_1+\tilde w(\cdot,s)$
changes sign for some $s$,  which is a contradiction.
Consequently, \eqref{eta-large} fails and
\cite[Proposition 4.4(i)]{PY05} 
implies
\begin{equation} \label{final-estA}
 \eta(s)=o(\xi(s)) \quad\hbox{and}\quad
   \xi(s)+\eta(s)=O(e^{\mu_0s})\qquad\hbox{as }\ s\to-\infty.
\end{equation}
The previous relations and \cite[Proposition 4.4(ii)]{PY05} 
further imply
\begin{equation} \label{final-estB} 
  \xi(s)=\tilde ce^{2\mu_0s}+o(e^{2\mu_0s})\quad\hbox{as}\quad s\to-\infty,
\end{equation}
where $\tilde c$ is a constant. We have $\tilde c\ne0$ due 
to \cite[(4.13)]{PY05}.
Consequently, there exists a constant $c\ne0$ such that
\begin{equation} \label{final-est}
|\xi_0(s)-ce^{\mu_0s}|+|\xi_1(s)|+\|\tilde w(\cdot,s)\|_0
 =o(e^{\mu_0s})\quad\hbox{as}\quad s\to-\infty.
\end{equation}
Since $\xi_0\vartheta_0+\xi_1\vartheta_1+\tilde w=w>0$, we have $c>0$
and estimate \eqref{final-est} yields
$$\|w(\cdot,s)-ce^{\mu_0s}\vartheta_0\|_0=o(e^{\mu_0s})\quad\hbox{as}\quad
s\to-\infty.$$
This completes the proof of Proposition \ref{prop-est} in the case
$p>p_H$. 

Next assume $11\leq N\leq 15$ and $p\in(p_L,p_H]$ (hence $\zeta\in(0,1/2)$).
Taking $s_2<s_1$ with $s_1\to-\infty$,  
 \eqref{estw2}, \eqref{est-w0} imply
\begin{equation} \label{ests1s2}
 c_0\int_{s_2}^{s_1}\|\tilde w\|_{1/2}^2\,ds 
 \leq \|\tilde w(s_2)\|_0^2 + o\Bigl(\int_{s_2}^{s_1}\psi\,ds\Bigr) 
 =o\Bigl(\psi(s_2) + \int_{s_2}^{s_1}\psi\,ds\Bigr).
\end{equation}
Hence, choosing any small $\ep\in(0,1)$, 
 \eqref{eq-xieta} guarantees that
$$\begin{aligned}
(1-\eps)\psi(s_2) &\leq \psi(s_1)-\mu_1\int_{s_2}^{s_1}\psi\,ds \leq \psi(s_1), \\
(1+\eps)\psi(s_2) &\geq \psi(s_1)-3\mu_0\int_{s_2}^{s_1}\psi 
    \geq \Bigl(1-\frac{3\mu_0}{1-\eps}(s_1-s_2)\Bigr)\psi(s_1),
\end{aligned}$$
provided $s_1$ is negative and sufficiently large. 
Consequently, there exists $\tau\in(0,1)$ (independent of $s_2$) 
such that for any sufficiently large negative $s_2$ we have
\begin{equation} \label{122}
 \frac12\psi(s_2)\leq\psi(s)\leq2\psi(s_2), \quad s\in(s_2,s_2+\tau).
\end{equation}
Now  \eqref{ests1s2} and \eqref{122} imply
\begin{equation} \label{est-intw}
 \int_{s}^{s+\tau}\|\tilde w\|_{1/2}^2=o\Bigl(\int_{s}^{s+\tau}\psi\,ds\Bigr)
\quad\hbox{as}\quad s\to-\infty.
\end{equation}

Set 
\begin{equation}
  \label{eq:6}
  \xi(s):=\int_s^{s+\tau}\xi_0^2(\sigma)\,d\sigma,
\qquad
\eta(s):=\int_s^{s+\tau}(\xi_1^2(\sigma)+\|\tilde w(\cdot,\sigma)\|_0^2)\,d\sigma.
\end{equation}
Notice also that for $\delta>0$ small enough, interpolation,
inequality $\|\tilde w\|_0\leq C\|\tilde w\|_{1/2}$
and \eqref{est-w0} imply 
\begin{equation} \label{interp-w}
\|\tilde w\|_\zeta^{2(1+\delta/2)}
\leq C\|\tilde w\|_0^{\delta}\|\tilde w\|_{1/2}^2
= o(\psi^{\delta/2})\|\tilde w\|_{1/2}^2. 
\end{equation}
Integrating \eqref{eq-xi12} with $i=0$ over the interval $(s,s+\tau)$ 
and using \eqref{122}, \eqref{est-intw}, and \eqref{interp-w}, we obtain
$$ \begin{aligned}
\dot\xi &=2\mu_0\xi+o\Bigl(\int_s^{s+\tau}(\|\tilde w\|_\zeta^2+\psi)^{1+\delta/2}\,d\sigma\Bigr) \\
 &=2\mu_0\xi+o(\psi^{1+\delta/2}) \\
 &=2\mu_0\xi+o((\xi+\eta)^{1+\delta/2}).
\end{aligned}$$
Similarly, integrating \eqref{eq-xi12} with $i=1$ and \eqref{estw2} we obtain
$$\dot\eta \leq 2\mu_1\eta+o((\xi+\eta)^{1+\delta/2}).$$
Again, \cite[Proposition 4.4]{PY05} guarantees 
the existence of $\tilde c\ne0$ such that
\eqref{final-estA} and \eqref{final-estB} hold, this time with $\xi$
and $\eta$ as in \eqref{eq:6}. 
Consequently,  there exist $c\ne0$ and $s_k\to-\infty$ such that 
$$
|\xi_0(s_k)-ce^{\mu_0s_k}|+|\xi_1(s_k)|+\|\tilde w(\cdot,s_k)\|_0
 =o(e^{\mu_0s_k})\quad\hbox{as}\quad k\to\infty.
$$
In addition, similarly as in the case $p>p_H$ 
we obtain $c>0$, hence  
\eqref{final-est3} is true.
\end{proof}

%----------------------------------------------------
\subsection{Radial monotonicity of entire solutions}
\label{subsec-monot}

We next establish the radial monotonicity of positive radial entire
solutions.

\begin{proposition} \label{prop-monot}
Assume $p>p_S$.
Let $u$ be a positive radial solution of \eqref{eq-entire}.
Then $u$ is radially decreasing.
\end{proposition}

\begin{proof}
Fix $T\in\R$ and let $v=v(\rho,s)$ be the rescaled function
corresponding to $u$ and $T$.
It is sufficient to prove that $v$ is radially decreasing.
Due to Subsection~\ref{subsec-ss},
the $\alpha$-limit set $\alpha(v)$ 
is a singleton $\{w\}$, where either $w=w_a$
with $a\in{\cal A}_k$, $k\geq2$, or $w=\phi_\infty$. 
The derivative $v_\rho$ solves a linear parabolic equation 
whose zero order coefficient is 
$$a_0(\rho,s):=-\frac{N-1}{\rho^2}-\frac12-\frac1{p-1}+pv^{p-1}(\rho,s).$$ 
Estimate \eqref{est-v} with $C_T=0$ guarantees the existence of $R_0>0$
such that $a_0(\rho,s)<-1/2$ when $\rho\geq R_0$.

Since $w_\rho<0$ for $\rho>0$  (cp.~\eqref{eq:2}), 
given $\eps>0$ there exists $s(\eps)$ such 
$v_\rho(\rho,s)<0$ for $\rho\in[\eps,R_0]$ and $s\leq s(\eps)$.
Assume $v_\rho(\rho_0,s_0)>0$ for some $\rho_0>R_0$ and $s_0\leq s(\eps)$.
Set $V:=v_\rho$ and let $D,D_0,\Omega(s),m(s),S$
be defined by \eqref{D},  \eqref{D1} and $\rho_1(s):= \inf\Omega(s)$.
Then $\rho_1(s)>R_0$ for $s\in(S,s_0]$ and
the same arguments as in the proof of Proposition~\ref{prop-comp}
yield $m(s_0)=0$, which is a contradiction.
Consequently, the maximum principle shows that $v(\cdot,s)$ 
is decreasing on $[\eps,\infty)$
for any $s\leq s(\eps)$.

If $w=w_a$ for some $a\in{\cal A}_k$ with $k\geq2$, then
$v(\cdot,s)\to w_a$ in $BC^1[0,\infty)$ as $s\to-\infty$,
hence there exist $\eps>0$ and $\tilde s_0\leq s(\eps)$ such that
$a_0(\rho,s)<-1$ if $\rho\in[0,\eps]$ and $s\leq \tilde s_0$.
If we assume $v_\rho(\rho_0,s_0)>0$ for some $\rho_0\in(0,\eps)$ and
$s_0\leq \tilde s_0$, then
the same arguments as above yield a contradiction.
Consequently,
$v(\cdot,s)$ is also nonincreasing on $[0,\eps]$ for $s\leq\tilde s_0$,
and the maximum principle guarantees that $v(\cdot,s)$ 
is decreasing on $[0,\infty)$ for all $s\in\R$.

Finally consider the case $w=\phi_\infty$
and assume on the contrary
that $v(\cdot,\tilde s_0)$ is not decreasing for some $\tilde s_0$.
Fix $\eps\in(0,1)$ such that
\begin{equation} \label{eps-choice}
 \frac{\kappa+1}{p-1}\eps<\frac{C_\eps}3, 
\quad\hbox{where}\quad C_\eps:=\frac{N-1}\eps-\frac\eps2,
\end{equation}
and 
\begin{equation} \label{phiinftykappa1}
 \phi_\infty(\eps)>\kappa+1.
\end{equation}
Then we can find $s_0\leq \tilde s_0$ such that for any $s\leq s_0$,
$v(\cdot,s)$ is decreasing on $[\eps,\infty)$ 
and $v(\cdot,s)$ attains a local minimum 
at some $\rho(s)\in[0,\eps)$.
Fix $\delta\in(0,1/2)$. For any $s\le s_0$, if
 $v(\rho(s),s)\le \kappa+\delta$, then the relations
 $v_\rho(\rho(s),s)=0\leq v_{\rho\rho}(\rho(s),s)$
and the equation for $v$
 imply $v_s(\rho(s),s)\geq \tilde\delta$ for some
$\tilde\delta>0$ (depending only on $p$ and $\delta$).
It follows that  there exists $s_1<s_0$ such that
\begin{equation} \label{kappadelta}
  \min_{[0,\eps]}v(\cdot,s)
  < \kappa+\delta\quad\hbox{for}\quad s\leq s_1.
\end{equation}
On the other hand, 
by \eqref{phiinftykappa1}
there exist $s_2\leq s_1$ such that
$v(\eps,s)>\kappa+1$ if $s\leq s_2$.

Let $z$ be the solution of the linear equation
$$z_s=z_{\rho\rho}+\frac{N-1}\rho z_\rho-\frac \rho2 z_\rho-\frac z{p-1} $$
in $(0,\eps)\times(0,\infty)$ 
satisfying the boundary conditions
$z_\rho(0,s)=0$, $z(\eps,s)=\kappa+1$,
and the initial condition $z(\rho,0)=0$.
Then $z$ is increasing in time and,
since $\kappa+1$ is a supersolution to $z$,
we obtain $z_\rho(\eps,s)\geq0$,
hence $z_\rho\geq0$ by the maximum principle.
Also, $z$ approaches a steady state $Z$
as $s\to\infty$ with $Z_\rho\geq0$, $Z_\rho(0)=0$, and  $Z(\eps)=\kappa+1$.
We have
$$\frac{\kappa+1}{p-1}\geq \frac{Z(\rho)}{p-1}
 =Z_{\rho\rho}+(\frac{N-1}{\rho}-\frac{\rho}{2})Z_\rho
 \geq Z_{\rho\rho}+C_\eps Z_\rho.$$
Integrating over $\rho\in(0,\eps)$
we obtain
$$\frac{\kappa+1}{p-1}\eps\geq Z_\rho(\eps)-Z_\rho(0)+C_\eps(Z(\eps)-Z(0))
   \geq C_\eps(\kappa+1-Z(0)),$$
hence \eqref{eps-choice} implies $Z(0)>\kappa+2/3$. 
Since $z_\rho\geq0$ and $z(0,s)\to Z(0)$ as $s\to\infty$,
we have $z(\cdot,S)>\kappa+1/2$ for some $S$ large enough.
Since
$v(\cdot,s_2-S+s)\geq z(\cdot,s)$ on $[0,\eps]$ for $s\in[0,S]$ by the comparison principle,
we obtain $v(\cdot,s_2)\geq \kappa+1/2$ on $[0,\eps]$, which contradicts
\eqref{kappadelta}.
 \end{proof}

%--------------------------------------------------------------
\subsection{Completion of the proof of Theorem~\ref{thmL1}}
\label{subsec-completion}
In this subsection we assume $p>p_L$.
By $\zeta(\cdot,\al)$ we will denote the solution of 
\begin{equation}
  \label{eq:1}
  \zeta_{\rho\rho}+\frac{N-1}r\zeta_\rho-\frac \rho2 \zeta_\rho-\frac 1{p-1}\zeta+\zeta^p=0
\end{equation}
with $\zeta(0,\al)=\al$. By \cite{M09},
 for each $\al>0$ there is $\rho_\al$ such that $\zeta(\cdot,\al)>0$ on
 $[0,\rho_\al)$ and $\zeta(\rho_\al,\al)=0$. Also, the following property 
is proved in \cite[Lemma 2.5]{M09} (although it is not stressed in  
\cite[Lemma 2.5]{M09}, it can be checked that the
constant $C$ there is independent of $\ep$):  
\begin{itemize} 
\item[(p1)]  For each compact interval $I\subset (0,\infty)$ one has 
  $\zeta(\cdot,\al)-\phi_\infty\to 0$ in $C^1(I)$ as $\al\to\infty$. 
\end{itemize}
We shall also need the following property of  $\zeta(\cdot,\al)$. 
\begin{lemma}  \label{le-on-z}
There is $\al_0>0$ such that 
for each $\al\ge \al_0$ one has
\begin{equation}
  \label{p1a}
  z_{[0,\rho_\al]}(\zeta(\cdot,\al)-\phi_\infty)\le 2.
\end{equation}
\end{lemma}

\begin{proof}
  Assume that, to the contrary, there are arbitrarily large values
  $\al$ with  $z_{[0,\rho_\al]}(\zeta(\cdot,\al)-\phi_\infty)
  \ge 3$. 
 Clearly, the zeros of $\zeta(\cdot,\al)-\phi_\infty$ are all simple
 and, since $\zeta(\rho_\al,\al)=0$, their number is even. Thus, there
must be at least 4 of them. We denote by
$\xi_1^\al<\dots <\xi_4^\al$ the first four zeros of
$\zeta(\cdot,\al)-\phi_\infty$. 

Let now $\mu_2$ be the third
eigenvalue of the linearization at  $\phi_\infty$
and $\vartheta_2$ a corresponding eigenfunction
(cp. (A3) in the  Subsection \ref{subsec-Ah}).
Then $\mu_2<0$,  $z(\vartheta_2)=2$, and
both zeros of $\vartheta_2$ are positive.
Let $\eta_1<\eta_2$ denote these zeros. 
As noted in \cite[Lemma 2.9]{M09}, a Sturm comparison argument
implies that $\eta_1\in (0,\xi_1^\al)$ and $\eta_2\in
(\xi_2^\al,\xi_3^\al)$.
Using (p1) and taking $\al$ sufficiently large we obtain that
\begin{equation*}
  p\zeta^{p-1}(\rho,\al)<p\phi_\infty^{p-1}(\rho)-\mu_2 
 \quad (\rho\in (\eta_1,\eta_2)). 
\end{equation*}
This relation and the fact that $\eta_1<\xi_1^\al<\xi_2^\al<\eta_2$ 
make the Sturm comparison argument applicable to the interval
$(\xi_1^\al,\xi_2^\al)$ as well.  
We conclude that this interval contains a
third zero of  $\vartheta_2$, which is a  contradiction.  
\end{proof}

Although it is not needed below, we remark that \eqref{p1a} 
in fact holds for all $\alpha>\kappa$. 
This follows from the observation that the zero number in \eqref{p1a} 
does not change as one varies $\alpha\in (\kappa,\infty)$  
(the fact that for $p>p_L$ one has $\rho_\alpha<\infty$ for all $\alpha>\kappa$ is important here).  
One can also turn the argument around and prove \eqref{p1a} 
by using the independence of the zero number of $\alpha$ 
in conjunction with the fact that the zero number is equal to 2 
for $\alpha>\kappa$ sufficiently close to $\kappa$ (see \cite[Lemma 2.3]{M09}).

\begin{proof}[Proof of Theorem~\ref{thmL1}]
  Let $u=u(r,t)$ be a positive bounded (radial)
  solution of \eqref{eq-entire} and $p>p_L$.
Fix $T\in\R$ and let $v$ be the corresponding 
rescaled solution of \eqref{eq-v-rad}.
We know from Subsection~\ref{subsec-ss} that
each of the sets $\alpha(v)$ and $\omega(v)$
has to be one of the sets $\{0\}$, $\{\kappa\}$ and $\{\phi_\infty\}$,
and $\alpha(v)\ne\{0\}$.
Estimate \eqref{est-v} (with $C_T=0$) guarantees $\alpha(v)\ne\{\kappa\}$.
Consequently, $\alpha(v)=\{\phi_\infty\}$.

We prove that
\begin{equation}
  \label{eq:3}
  z(v(\cdot,s)-\phi_\infty)\le 1 \quad(s\in \R).
\end{equation}
In fact, assume there is $s_0$ such that
  $z(v(\cdot,s)-\phi_\infty)\ge 2$ for $s=s_0$ (hence for all $s\le
  s_0$). Making $s_0$ smaller if needed we may assume that the first
  two zeros $\xi_1$, $\xi_2$ of  $v(\cdot,s_0)-\phi_\infty)$ are
  simple. Clearly, $\xi_2$ being the second zero, there is
  $\bar\xi>\xi_2$ such that $v(\bar\xi,s_0)<\phi_\infty(\bar\xi)$. Using (p1)
  and Lemma \ref{le-on-z}, 
  we find $\al$ such that \eqref{p1a} holds along with the following
  statements   
  \begin{itemize}
  \item[(a1)] $v(\cdot,s_0)-\zeta(\cdot,\al)$ has  zeros $\tilde
    \xi_1$,  $\tilde \xi_2$ (near $\xi_1$, $\xi_2$, respectively) 
    with $\tilde \xi_1<\tilde \xi_2<\bar \xi$,
   \item[(a2)] $v(\bar\xi,s_0)<\zeta(\bar\xi,\al)$.
  \end{itemize}
Relations $v>0$, $\zeta(\rho_\al,\al)=0$,  and (a2)  imply that 
$v(\cdot,s_0)-\zeta(\cdot,\al)$ has another zero in $(\bar\xi,\rho_\al)$. 
Thus, $z_{[0,\rho_\al]}(v(\cdot,s)-\zeta(\cdot,\al))\ge 3$ for all $s\le
  s_0$. 

  Let $\al$ be as above. 
Proposition~\ref{prop-monot} and the convergence 
$v(\cdot,s)\to \phi_\infty$
in $C^1_{loc}(0,\infty)$ imply that
there exist $s_1\in\R$ and $\de>0$ such that
$v(\rho,s)>\zeta(\rho,\al)$ for all $\rho\in[0,\delta]$ and
$s\in(-\infty,s_1)$.
Using this relation (and the convergence again), 
we obtain that for  
all sufficiently large
negative $s$ we have
\begin{equation*}
 z_{[0,\rho_\al]}(v(\cdot,s)-\zeta(\cdot,\al))=
z_{[\de,\rho_\al]}(v(\cdot,s)-\zeta(\cdot,\al))\le
z_{[\de,\rho_\al]}(\phi_\infty-\zeta(\cdot,\al))\le 2
\end{equation*}
(cp. \eqref{p1a}). This contradiction completes the proof of \eqref{eq:3}.

We now show that the case  $z(v(\cdot,s)-\phi_\infty)= 1$ for some
$s$ is impossible. 
Indeed, if this holds, then 
$z(u(\cdot,t_0)-\phi_\infty)= 1$ for some $t_0$. Setting 
$u_0(r):=\min \{u(r,t_0),\phi_\infty(r)\}$, we have 
$0<u_0\le\phi_\infty$ and $\phi_\infty-u_0$ has compact support. 
By \cite{PY03}, the solution of 
$\bar u_t=\Delta \bar u+\bar u^p$, $\bar u(\cdot,t_0)=u_0$  is unbounded 
(it approaches $\phi_\infty$), and the comparison principle then implies
that the same it true of $u(\cdot,t)$, in contradiction to our
assumption. 

Thus $z(v(\cdot,s)-\phi_\infty)= 0$, that is, 
$v(\cdot,s)<\phi_\infty$, for all $s\in \R$.

To complete the proof, we now apply some  results of \cite{PY05}.
Recall from Subsection~\ref{subsec-Ah} that 
$\vartheta_0(\rho)=\hat c_0\rho^\beta$,
where $\hat c_0>0$, and $\|\cdot\|_0$ denotes
the norm in $L^2(0,\infty;\rho^{N-1}e^{-\rho^2/4}d\rho)$
The steady states $\phi_\alpha$ satisfy
$$ \phi_\alpha(r)=Lr^{-2/(p-1)}-b(\alpha)r^\beta+O(r^{\beta-\eps})
\quad\hbox{as}\quad r\to\infty,$$
where $\eps>0$, $b(\alpha)=b_1\alpha^{1+\beta(p-1)/2}$, and $b_1>0$ is
a constant; see \cite{GNW, W93}.
According to \cite[Lemma 2.2]{PY05}, the rescaled functions
$$ \psi_\alpha(\rho,s)=e^{-s/(p-1)}\phi_\alpha(e^{-s/2}\rho) $$
(cf.~\eqref{vu}) satisfy
\begin{equation} \label{estPY}
 \|\phi_\infty-\psi_\alpha(\cdot,s)
 -\frac{b(\alpha)}{\hat c_0}e^{\mu_0s}\vartheta_0\|_0
 =o(e^{\mu_0s}) \quad\hbox{as}\quad s\to-\infty.
\end{equation}
Fix $\alpha$ such that $b(\alpha)=c\hat c_0$,
where $c$ is from \eqref{final-est3}.
Then \eqref{estPY} and Proposition~\ref{prop-est} 
imply
$$ \|v(\cdot,s_k)-\psi_\alpha(\cdot,s_k)\|_0=o(e^{\mu_0s_k})
 \quad\hbox{as}\quad k\to\infty.$$
As shown in \cite[Lemma 4.2]{PY05}, this estimate
guarantees that $v\equiv\psi_\alpha$, hence $u\equiv\phi_\alpha$.
\end{proof}

%--------------------------------------------------------------
\section{Proofs of Theorems~\ref{thm1} and~\ref{thm2}}
\label{sec-proofs2}

In the proofs of Theorems \ref{thm1}, \ref{thm2}, we will use
the following result. 
\begin{proposition} \label{prop1}
  Let $p>p_{S}$
  and $u$ be a positive radial solution of
  \eqref{eq-ancient} satisfying \eqref{u-Tdecay}.
  Then there is a positive integer $m$ such that
  $z(u(\cdot,t)-\phi_\alpha)\le m$ for all $t<T$ and
  $\alpha\in(0,\infty]$. 
\end{proposition}

\begin{proof}
  Fix any $\alpha\in(0,\infty]$ and set $\phi:=\phi_\alpha$,
  \begin{equation}
    \label{eq:5}
    {v}(\rho,s):=e^{-s/(p-1)}{u}(e^{-s/2}\rho,T-e^{-s}), \ \
{\psi}(\rho,s):=e^{-s/(p-1)}{\phi}(e^{-s/2}\rho)
  \end{equation}
(in particular, $\psi\equiv\phi_\infty$ when $\alpha=\infty$). 
Then $v,\psi$ solve equation \eqref{eq-v-rad}.
By  \eqref{u-Tdecay},
there is $s_0\in\R$ such that
$v$ is bounded for $s\leq s_0$.

Remarks in Subsection~\ref{subsec-ss} 
show that for the $\alpha$-limit set 
of $v$  in $C^1_{loc}[0,\infty)$ we have either
$\alpha(v)=\{w_a\}$ with
$a\in{\cal A}_k$ for some $k>0$ or 
$\alpha(v)=\{\phi_\infty\}$. The latter is ruled out
by the boundedness of $v$ for $s\leq s_0$, so we have the former.
We prove that the conclusion of the proposition holds with $m=k+1$
(which is independent of $\al$). 
Note that $\alpha(\psi)=\{\phi_\infty\}$ in $C^1_{loc}(0,\infty)$.

Since $\phi(r)\leq Cr^{-2/(p-1)}$,
we also have $\psi(\rho,s)\leq C\rho^{-2/(p-1)}$ for all $s$
and we can fix $R_0>\phi_\infty^{-1}(\kappa)$ such that 
\begin{equation} \label{psic0}
\psi(\rho,s)\leq c_0
\quad\hbox{for all}\quad s\ \hbox{ and }\ \rho\geq R_0,
\end{equation} 
where $c_0$ is defined in \eqref{est-f2}.

Consider the function $V:=v-\psi$.
The $k$ zeros of $w_a-\phi_\infty$ belong to the interval
$(0,\phi_\infty^{-1}(\kappa)]\subset(0,R_0)$ 
(cp.~\eqref{eq-a}).
Also, since $w_a,\phi_\infty$ solve the same second order ODE, the
zeros are simple.
This fact and the convergence of $v,\psi$ as $s\to-\infty$
guarantee that, decreasing $s_0$ if necessary, we have
\begin{equation} \label{zVk}
z_{(0,R_0)}(V(\cdot,s))=k \quad\hbox{for all}\quad s\leq s_0.
\end{equation}

Assume, for a contradiction, that
$z((V(\cdot,s_1))\geq k+2$ for some $s_1\leq s_0$.
Decreasing $s_0$ further if needed, we may assume $s_1=s_0$.
Denoting by $\rho^*$ the $(k+1)$-th zero of $V(\cdot,s_0)$,
we can choose $\rho_0>\rho^*$ such that $V(\rho_0,s_0)<0$.
Let $D,D_0,\Omega(s),m(s),S$ be as in \eqref{D}, \eqref{D1},
and $\rho_1(s):= \inf\Omega(s)$. Clearly, $\rho_1(s)$ is a zero of
$V(\cdot,s)$ for each $s\in(S,s_0]$, and, by the monotonicity
of the zero number, $\rho_1(s)$ is at least $(k+1)$-th zero of
$V(\cdot,s)$. Hence \eqref{zVk} implies $\rho_1(s)\geq R_0$
for $s\in(S,s_0]$. Now \eqref{psic0} and Proposition~\ref{prop-comp}
give $V(\rho_0,s_0)>0$, 
and we have a contradiction.

Consequently, $z((V(\cdot,s))\leq k+1$
for all $s\leq s_0$ and the monotonicity
of the zero number gives the same estimate for $s>s_0$.
This gives the desired estimate
 $z(u(\cdot,t)-\phi_\alpha)\le k+1$.
\end{proof}

%--------------------------------------------------
\begin{proof}[Proof of Theorem~\ref{thm1}]
  By standard results, since the solution $u$ is bounded,  its
  $\omega$-limit set in $C^1_{loc}([0,\infty))$, $\omega(u)$, is a
nonempty compact set in this space and the desired conclusion 
$u(\cdot,t)\to 0$ is equivalent to $\omega(u)=\{0\}$.
Also, $\omega(u)$ is invariant: for any $u_\infty^0\in\omega(u)$ 
there is a radial  solution of \eqref{eq-entire}
satisfying $u_\infty(\cdot,0)=u_\infty^0$ and $u_\infty(\cdot,t)\in
\omega(u)$ for all $t\in\R$. Obviously, any such $u_\infty$
is nonnegative and bounded. 
  
Set
$$\ell^-:=\liminf_{t\to\infty}u(0,t)\quad\hbox{and}\quad
\ell^+:=\limsup_{t\to\infty}u(0,t).$$
By the boundedness of $u$, these limits are finite. 
We first prove that $\ell^+=\ell^-$. 
Assume not and fix $\alpha\in(\ell^-,\ell^+)$.
Then $u(0,t_k)=\alpha=\phi_\alpha(0)$ 
(and $u_r(0,t_k)=0=\phi'_\alpha(0)$)
for an infinite sequence 
$t_k\to\infty$.
It follows that $z(u(\cdot,t)-\phi_\alpha)$ drops at each $t_k$
(cp.~Proposition~\ref{prop-zero}),
which is a contradiction to Proposition~\ref{prop1}.
Thus, $\ell^+=\ell^-=:\alpha$,
which implies that $u(0,t)\to\alpha$ as $t\to\infty$.

Consequently, any element $u_\infty^0$ of $\omega(u)$
has $u_\infty^0(0)=\alpha=\phi_\alpha(0)$. We show that actually
$u_\infty^0\equiv \phi_\alpha$. Assume that, to the contrary,
$u_\infty^0(r_0)\ne \phi_\alpha(r_0)$ for some $r_0>0$.
Let $u_\infty$ be the entire solution of \eqref{eq-entire}
corresponding to $u_\infty^0$, as above.
Then $u(0,t)=\alpha$ (and $u_r(0,t)=0$)
for all $t$, and 
 $u_\infty(r_0,t)\ne \phi_\alpha(r_0)$ for $t\approx 0$. 
Hence $z_{(0,r_0)}(u_\infty(\cdot,t)- \phi_\alpha)$
is finite for $t$ near $0$ and drops at any such $t$, which is
absurd. Thus we have showed that $\omega(u)=\{\phi_\alpha\}$.
 
To conclude, assume $\alpha>0$ and fix $\beta>0$, $\beta\ne\alpha$.
Then $z(u(\cdot,t)-\phi_\beta)$ is bounded by Proposition~\ref{prop1}.
However, in the considered range  $p_S<p<p_{JL}$ we have 
$z(\phi_\alpha-\phi_\beta)=\infty$ (see \cite{W93, QS07})
and the zeros of $\phi_\alpha-\phi_\beta$ are simple.
The convergence of $u(\cdot,t)$ to $\phi_\alpha$ therefore implies
that $z(u(\cdot,t)-\phi_\beta)\to \infty$ as $t\to\infty$, a
contradiction. 
Thus, $\alpha=0$ and we have proved the desired conclusion $\omega(u)=\{0\}$. 
\end{proof}
%----------------------------------------------------------
\begin{proof}[Proof of Theorem~\ref{thm2}]
  Assume that $v$ is not a steady state of \eqref{eq-v}.
We know from Subsection~\ref{subsec-ss}
that each of the sets $\alpha(v)$, $\omega(v)$ is
a singleton consisting of either $w_a$ for some $a\in{\cal A}$,
or $\phi_\infty$. 
In addition, $\alpha(v)\ne\{0\}$ and monotonicity of the energy
functional (cp. Proposition \ref{prop-energy}) gives
$\omega(v)\ne\alpha(v)$.
Obviously, $\alpha(v)=\{\phi_\infty\}$ if and only if \eqref{u-Tdecay}
fails; if \eqref{u-Tdecay} holds, 
we necessarily have $\alpha(v)=\{w_a\}$ for some $a\in{\cal A}$.

We next prove that  $\omega(v)=\{\tilde w\}$ where $\tilde w= w_a$
for some $a\in{\cal A}$ (possibly $a=0$).
For that, we just need to show that 
$\omega(v)\ne\{\phi_\infty\}$.
If $p>p_L$, this follows from 
Proposition~\ref{propEinftykappa}, as already noted in
Subsection~\ref{subsec-ss} (thus,
$\omega(v)=\{0\}$ or $\omega(v)=\{\kappa\}$ in this case).
If $p_S<p<p_{JL}$ and \eqref{u-Tdecay} fails, then the relation
follows from  $\{\phi_\infty\}=\alpha(v)\ne \omega(v)$.

If  $p_S<p<p_{JL}$ and 
\eqref{u-Tdecay} is true,  Proposition~\ref{prop1} applies. 
Let $m$ be as in the proposition.
Suppose $\omega(v)=\{\phi_\infty\}$. Since
$z(\phi_\infty-\phi_\alpha)=\infty$ for any $\alpha>0$
(see \cite{W93} or \cite{QS07}),  
for all sufficiently large $s$, the function $v(\cdot,s)-\phi_1$
has at least $m+1$  zeros. Pick any such $s$ and set
$\alpha:=e^{s/(p-1)}$. 
By the scaling invariance of equation
\eqref{eq-Fuj}, we can write
$\phi_1(\rho)=\alpha^{-1}\phi_\alpha(\alpha^{-(p-1)/2}\rho)$.
Using this and the relation between $u$ and $v$ (cp. \eqref{eq:5}),
we obtain,  for $t=T-e^{-s}$,
\begin{equation*}
 m+1\le  z(v(\cdot,s)-\phi_1)=z(v(\cdot,s)-e^{-s/(p-1)}{\phi_\alpha}(e^{-s/2}\rho)))=z(u(\cdot,t)-\phi_\alpha), 
\end{equation*}
and we have a contraction to Proposition \ref{prop1}.

To complete the proof of Theorem~\ref{thm2}, it remains to show
that $\omega(v)=\{\tilde w\}$ implies the convergence
\begin{equation}\label{connection-tws}
  \lim_{s\to\infty}v(\cdot,s)=\tilde w
\end{equation}
in  $C^1_{loc}(\R^N)$ (and not just in
$C^1_{loc}(\R^N\setminus\{0\})$, the space used in the definition of
$\omega(v)$),
and that $\alpha(v)=\{w\}$ in conjunction with \eqref{u-Tdecay}
implies the convergence 
\begin{equation}\label{connection-ws}
  \lim_{s\to-\infty}v(\cdot,s)=w
\end{equation}
in  $C^1_{loc}(\R^N)$.
The latter is a simpler: \eqref{connection-ws} follows from the convergence in
$C^1_{loc}(\R^N\setminus\{0\})$,   
the boundedness of $v(\cdot,s)$ as $s\to-\infty$ (condition \eqref{u-Tdecay}),
and parabolic estimates. 

The former can be proved similarly once we show that as $s\to\infty$
the function $v(\cdot,s)$ stays bounded on a neighborhood of the
origin. For this, we use a ``no-needle'' lemma, Lemma 2.14 of \cite{MM09}.
Consider the functions $v(\cdot,k+\cdot)$, $k=1,2, \dots$.
Since the sequence $v(\cdot,k)$ converges in
$C^1_{loc}(\R^N\setminus\{0\})$  to $\tilde w$, a bounded 
function,  \cite[Lemma 2.14]{MM09} yields positive constants $\de>0$,
$M_1$ such that $v(\cdot,k+\de)\le M_1$ on $\R^N$ for $k=1,2,\dots$.
Making $M_1$ larger if necessary, we may also assume that $\tilde
w(0)<M_1$.  
Take now any $a>M_1$ and let $w_a$ be the solution of
\begin{equation}
 \label{eq:7s}
\begin{aligned}
  w_{\rho\rho}+\frac{N-1}\rho w_\rho-\frac \rho2 w_\rho-\frac
  w{p-1}+w^p&=0,\quad \rho >0,\\
  w(0)=a,\quad w'(0)&=0.
\end{aligned}\end{equation}
Then $w_a$ is defined (at least) on a small interval $[0,R]$ and,
making $R>0$ smaller if needed, we have
$w_a>M_1 >\tilde w$ on $[0,R]$. Since $v(R,s)\to \tilde w(R)$,
$v(R,s)<w_a(R)$ for all sufficiently large $s$. Since also
$v(r,k+\de)\le M_1<w_a(r)$ for all $r\in [0,R]$ and $k=1,2,...$, we
obtain from the comparison principle that $v(r,s)<w_a(r)$
for all $r\in [0,R]$ if $s$ is large enough. This is the desired
estimate, from which \eqref{connection-tws} is proved easily.
\end{proof}

We remark that the monotonicity of $s\mapsto z(v(\cdot,s)-\phi_\infty)$
implies that the steady states $w$ and $\tilde w$ in
\eqref{connection-w} satisfy
$z(w-\phi_\infty)\geq z(\tilde w-\phi_\infty)$.

%--------------------------------------------------------------
\section{Further results and applications} \label{sec-appl}
In the following theorem, we consider two classes of positive radial
solutions of \eqref{eq-Fuj} for $p>p_L$. The first class consists of solutions which exhibit a type II blowup and the second class of
global  solutions which decay to 0 with rate  slower than
$t^{-1/(p-1)}$, or do not decay at all. As an application of our new
Liouville theorem, Theorem \ref{thmL1}, we show that at least along a
sequence of times, the profiles of the solutions have a limit.
\begin{theorem} \label{thm-profile}
Let $p>p_L$ and $u$ be a positive radial solution of \eqref{eq-Fuj} in $\R^N\times(0,T)$.
Assume that 
$$
\begin{aligned}
\hbox{either \ }&T<\infty,\ \
\limsup_{t\to T}(T-t)^{\frac1{p-1}}\|u(\cdot,t)\|_\infty=\infty
\hbox{ \ (type II blowup)}, \\
\hfill\break
\hbox{or \ }&T=\infty, \ \
\limsup_{t\to T}t^{\frac1{p-1}}\|u(\cdot,t)\|_\infty=\infty
\hbox{ \ (slow or no decay)}.
\end{aligned}
$$
Then there exist $t_k\to T$ such that
\begin{equation} \label{eq-profile}
 \lambda_k^{2/(p-1)} u(\lambda_k r,t_k)\to \phi_1(r),
\qquad \lambda_k:=\frac1{\|u(\cdot,t_k)\|_\infty^{(p-1)/2}},
\end{equation}
uniformly in $r\in[0,\infty)$.
\end{theorem}
In the blowup case, this theorem is 
known to hold for any $p\geq p_S$ 
under the extra assumption that $z(u_t(\cdot,t))<\infty$:
a long and technically involved proof can be found in
\cite{MM04}. 
The convergence in \eqref{eq-profile} plays a 
key role in \cite{MM04} in the study of blowup rates and profiles.     
\begin{proof}[Proof of Theorem \ref{thm-profile}]
The proof is based on doubling, scaling, one-dimen\-sion\-al Liouville
theorem, and Theorem~\ref{thmL1}. 
 
Considering equation \eqref{eq-Fuj} on the time interval
$(\delta,T)$ instead of $(0,T)$ (where $0<\delta<T$)
we may assume that 
\begin{equation} \label{uBdd0}
\hbox{$\|u(\cdot,t)\|_\infty$ is bounded
for $t\in(0,\tau)$ whenever $\tau<T$.}
\end{equation} 

Set 
$$ M(t):=\|u(\cdot,t)\|_\infty^{(p-1)/2}.$$
Our assumptions imply that there exist $t_k\to T$
such that $M_k:=M(t_k)>2k/d(t_k)$,
where $d(t):=\min(\sqrt{t},\sqrt{T-t})$ ($d(t)=\sqrt{t}$ if $T=\infty$).
The Doubling Lemma \cite[Lemma~5.1]{PQS07a} guarantees that,
possibly after modifying the sequence $\{t_k\}$, the following
additional condition is satisfied for $k=1,2,\dots$:
$M(t)\leq 2M_k$ whenever  $\sqrt{|t-t_k|}<k/M_k$.

Set $\lambda_k:=1/M_k$.
We claim that given any $\eps_0\in(0,1)$  
 there exists $R_0=R_0(\eps_0)$
such that for a suitable subsequence of $k$ we have
$u(r,t_k)<\eps_0\|u(\cdot,t_k)\|_\infty$
whenever $r/\lambda_k>R_0$.
Assume that no such $R_0$ exist.
Then we can find (a subsequence of $k$ still denoted by $k$ and) $r_k$ such that
$r_k/\lambda_k\to\infty$ and $u(r_k,t_k)\geq\eps_0\|u(\cdot,t_k)\|_\infty$.  
Set 
$$U_k(\rho,s):=\lambda_k^{2/(p-1)}u(r_k+\lambda_k\rho,t_k+\lambda_k^2s).$$
Then for $k=1,2,\dots$, $U_k$ satisfies the equation
$$  U_s = U_{\rho\rho}+\frac{N-1}{r_k/\lambda_k+\rho}U_\rho+U^p, $$
$U_k$ is bounded in
$\{(\rho,s):\sqrt{|s|}<k,\ \rho\geq-r_k/\lambda_k \}$
by a constant independent of $k$, and
$U_k(0,0)\geq\eps_0$.
Since $r_k/\lambda_k\to\infty$,
(a suitable subsequence of)
$\{U_k\}$ converges 
to a positive solution of \eqref{eq-entire} with $N=1$,
which contradicts the corresponding Liouville theorem
(see the second part of Theorem \ref{thm-subcrit}).
The claim is thus proved.

Take now a decreasing sequence $\eps_j\to0$. Using
a diagonalization argument, we find a subsequence of $k$  such that
$u(r,t_k)<\eps_j\|u(\cdot,t_k)\|_\infty$
whenever $r/\lambda_k>R_j:=R_0(\eps_j)$ and $k$ is large enough.

Next set
\begin{equation} \label{eq-Vk}
V_k(\rho,s):=\lambda_k^{2/(p-1)}u(\lambda_k\rho,t_k+\lambda_k^2s).
\end{equation}
Then $V_k$ satisfies the equation
$$  V_s = V_{\rho\rho}+\frac{N-1}{\rho}V_\rho+V^p, $$
$V_k$ is bounded by 2 in $\{(\rho,s):\rho>0,\ \sqrt{|s|}<k\}$,
and $V_k(\cdot,0)$ attains its maximum 1 in the compact set 
$\{\rho:|\rho|\leq R_0\}$ (since $V_k(\rho,0)\leq\eps_0$ for $\rho>R_0$).
A suitable subsequence of
$\{V_k\}$ converges (in $C_{loc}$, for example)
to a positive solution of \eqref{eq-entire},
hence to a steady state $\phi_\alpha$.
Since $\max V_k(\cdot,0)=1$, we have $\alpha=1$.
Since $V_k(\rho,0)\leq\eps_j$ for $\rho>R_j$ and $k$ large enough,
we see that the convergence is uniform on $[0,\infty)$.
\end{proof}

We now return to the classification problem for entire solutions satisfying 
\eqref{u-Tdecay} (cp. Theorem \ref{thm1}). As mentioned in the
introduction, we believe that the statement of Theorem \ref{thm1}
holds also in the range $p_{JL}\le p<p_{L}$. We can actually prove
this, see Proposition \ref{prop-pJLL} below,
provided the following condition on the energies
of steady states of \eqref{eq-v-rad} is satisfied:
\begin{equation} \label{Ewa-n}
E(w_a)<E(\phi_\infty)\quad \text{for all}\quad a\in{\cal A}. 
\end{equation}
This looks plausible, although the proof may not be easy. One way
\eqref{Ewa-n} could be verified is by proving the existence of a
solution of  \eqref{eq-v-rad} connecting $\phi_\infty$ to $w_a$, for any
$a\in \cal A$. Then the monotonicity of the energy would give \eqref{Ewa-n}
immediately. The question whether such connections indeed exist is
of independent interest. A positive answer would give an interesting
information on the variety of entire solutions of \eqref{eq-v-rad}.
What seems to be crucial for establishing the connections is
a description of the (global) bifurcation diagram for the
steady states of \eqref{eq-v-rad} when $p$ decreases from $p_L$
down to $p_{JL}$. Optimally, one would prove that all  
classical steady states lie on bifurcation branches
emanating from the singular steady state at some bifurcation values of
$p$. If this could be proved, then there is hope that the connections can
first be established locally, near bifurcation points, then globally
by continuation, somewhat in the spirit of  \cite[Section~3]{FMP02}. 

\begin{proposition} \label{prop-pJLL}
Let $p_{JL}\le p<p_L$.
Assume \eqref{Ewa-n}.
If $u$ is a positive radial bounded solution of
\eqref{eq-entire} satisfying \eqref{u-Tdecay},
then \ $\lim\limits_{t\to\infty}\|u(\cdot,t)\|_\infty=0$
(i.e.~$u$ is a homoclinic solution to the trivial steady state).
\end{proposition}

\begin{proof}
   Assume $p_{JL}\le p<p_L$. 
Let $u$ be a positive radial bounded solution of
\eqref{eq-entire} satisfying \eqref{u-Tdecay}
and let $C$ be the constant from \eqref{u-Tdecay}.
Set ${\cal A}^C:={\cal A}\cap[0,C]$.
As proved in \cite{PQ19}, 
  the set ${\cal A}^C$ is finite.
  Using \eqref{Ewa-n}, we find   $\eps>0$ such that
\begin{equation} \label{epsE}
\eps<E(\phi_\infty)-\max_{a\in{\cal A}^C}E(w_a).
\end{equation}

Using   
Proposition~\ref{prop1} and the same arguments as in the
proof of Theorem~\ref{thm1}, one shows that
$\omega(u)=\{\phi_\alpha\}$ for some $\alpha\in[0,\infty)$.
We need to prove that $\alpha=0$.
Suppose for a contradiction that $\alpha>0$.
Set $u_k(r,t):=k^{2/(p-1)}u(kr,k^2t)$.
Then $u_k$ is a positive radial bounded solution of
\eqref{eq-entire} satisfying \eqref{u-Tdecay} and
$\omega(u_k)=\{\phi_{k^{2/(p-1)}\alpha}\}$. Notice that
$\phi_{k^{2/(p-1)}\alpha}\nearrow\phi_\infty$ as $k\to\infty$.
Therefore, we can find $k$ and $T$ such that
$E(u_k(\cdot,T-1))>E(\phi_\infty)-\eps$.
Let $v_k$ be the rescaled function corresponding to $u_k$ and $T$.
Then $v_k(\cdot,0)=u_k(\cdot,T-1)$,
hence $E(v_k(\cdot,0))>E(\phi_\infty)-\eps$.
Assumption \eqref{u-Tdecay} guarantees that
$$
\begin{aligned}
|v_k(0,-\log(T-t))| &=(T-t)^{-1/(p-1)}k^{2/(p-1)}u(0,k^2t) \\
&\leq C(T-t)^{-1/(p-1)}k^{2/(p-1)}(k^2t)^{-1/(p-1)} \\
&\to C \ \hbox{ as }\ t\to-\infty,
\end{aligned}
$$
hence $\alpha(v_k)=\{w_a\}$ for some $a\in{\cal A}^C$,
$E(w_a)> E(v_k(\cdot,0))>E(\phi_\infty)-\eps$ and we have a
contradiction to 
\eqref{epsE}.
\end{proof}

\begin{remark}
  \label{rm:added}
  {\rm Condition \eqref{Ewa-n} is also sufficient for the
    validity of Theorem~\ref{thm2} for $p_{JL}\le p\le p_L$
    (cp. Remark \ref{rm:en}). Indeed, the proof of
    Theorem \ref{thm2} as given above
    applies in the case $p_{JL}\le p\le p_L$ with a single exception
    of the argument we used for proving the relation
    $\omega(v)\ne\{\phi_\infty\}$ 
    in the case that $\alpha(v)=\{w_a\}$ for some $a\in\cal
    A$. Obviously, if  \eqref{Ewa-n} holds, then instead of that
    argument one can simply refer to the monotonicity of the energy
    functional.  }
\end{remark}

\noindent{\bf Acknowledgment.} A major part of this research was done 
during visits of the second author at the University of Minnesota 
and the first author at the Comenius University. We thank
the mathematics departments at these universities for the hospitality. 
%---------------------------- 
\def\by{\relax} 
\def\paper{\relax} 
\def\jour{\textit} 
\def\vol{\textbf} 
\def\yr#1{\rm(#1)} 
\def\pages{\relax} 
\def\book{\textit} 
\def\inbook{In: \textit} 
\def\finalinfo{\rm} 
%----------------------------

\bibliographystyle{amsplain}

\end{document}